\documentclass[11pt]{amsart}
\usepackage{graphicx}
\usepackage{latexsym}
\usepackage{amsfonts}
\usepackage[utf8]{inputenc}
\usepackage{amsthm}
\usepackage{amsmath}
\usepackage{amssymb}
\usepackage[all]{xy} 
\usepackage{mathrsfs} 
\usepackage{stmaryrd} 
\usepackage{eepic}
\usepackage{epic}
\usepackage{eucal}
\usepackage{epsfig}
\usepackage{psfrag}
\usepackage{float}
\usepackage{amssymb}
\usepackage{amsmath}
\usepackage{amsthm}
\usepackage{amsbsy}
\usepackage{bm}
\usepackage{graphicx}
\usepackage{color}
\usepackage[ps2pdf,breaklinks=true,colorlinks=true]{hyperref}
\usepackage[colorlinks=true]{hyperref}
\usepackage{mathrsfs} 
\usepackage{psfrag,epsf}
\usepackage{color}
\usepackage{graphicx}
\usepackage{amsmath}
\usepackage{amsfonts}
\usepackage{enumerate}
\usepackage{latexsym}
\usepackage{amsthm}

\makeatletter\@addtoreset{equation}{section}\makeatother
\makeatletter\@addtoreset{figure}{section}\makeatother
\makeatletter\@addtoreset{table}{section}\makeatother

\newtheorem{theorem}{Theorem}[section]
\newtheorem{hyp}[theorem]{Hypothesis}
\newtheorem{prop}[theorem]{Proposition}
\newtheorem{remark}[theorem]{Remark}
\newtheorem{lemma}[theorem]{Lemma}

\newtheorem{cor}[theorem]{Corollary}

\newtheorem{theoremL}{Theorem}

\newcommand{\bE}{{\mathbb E}}
\newcommand{\bR}{{\mathbb R}}

\newcommand{\bZ}{{\mathbb Z}}
\newcommand{\bQ}{{\mathbb Q}}

\newcommand{\bT}{{\mathbb T}}
\newcommand{\bD}{{\mathbb D}}
\newcommand{\bP}{{\mathbb P}}
\renewcommand{\i}{\infty}
\renewcommand{\a}{\alpha}

\renewcommand{\d}{\delta}
\newcommand{\D}{\Delta}
\newcommand{\e}{\varepsilon}
\newcommand{\g}{\gamma}
\newcommand{\G}{\Gamma}
\renewcommand{\l}{\lambda}
\renewcommand{\L}{\Lambda}

\newcommand{\s}{\sigma}

\renewcommand{\O}{\Omega}
\renewcommand{\th}{\theta}

\renewcommand{\o}{\omega}
\newcommand{\z}{\zeta}

\newcommand{\cS}{{\mathcal S}}
\newcommand{\cX}{{\mathcal X}}
\newcommand{\cI}{{\mathcal I}}
\newcommand{\cL}{{\mathcal L}}
\newcommand{\cF}{{\mathcal F}}
\newcommand{\cT}{{\mathcal T}}
\newcommand{\cG}{{\mathcal G}}

\renewcommand{\geq}{\geqslant}
\renewcommand{\leq}{\leqslant}

\newcommand{\op}[1]{\!\!\mathop{\rm ~#1}\nolimits}

\newenvironment{example}{\refstepcounter{theorem}\par\medskip\noindent{\bf
Example~\thetheorem~~}}{\unskip\nobreak\hfill\hbox{}}

\newenvironment{definition}{\refstepcounter{theorem}\par\medskip\noindent{\bf
Definition~\thetheorem.}}{\unskip\nobreak\hfill\hbox{\medskip \smallskip}}

\begin{document}

\title[Poincar\'e\--Birkhoff theorems in random dynamics]{Poincar\'e\--Birkhoff  Theorems
in random dynamics}

\author{\'Alvaro Pelayo \,\,\,\,\,\,\,\,\,\,\,\,\,\, \,\,\,\,\,Fraydoun Rezakhanlou}

\maketitle

\vspace{-10mm}
\begin{center}
{\small {{\bf To Alan Weinstein
 on his 70th birthday, with admiration.}}}
\end{center}

\begin{abstract}
We propose a generalization of the Poincar\'e\--Birkhoff  Theorem on area\--preserving  
twist maps  to  area\--preserving twist maps that are random with respect to an ergodic probability 
 measure. The classical theory is a particular instance of the random theory we
 propose.
  \end{abstract}

\section{\textcolor{black}{Introduction and main results}} \label{mainsec} 

This paper proposes \emph{an} extension of the classical theory of area\--preserving twist maps to the random
setting. While of course there is not a unique way to do this, 
our definitions and constructions are natural both from the point of view of probability and the point of view of geometry.  This is evidenced by the fact that the classical theory of area preserving 
twist maps is a particular case of the random theory, as explained in Appendix B. Nonetheless we look
forward to seeing complementary approaches in the literature where other notions of random
twists may be considered.

In his work in celestial mechanics~\cite{Po93}
Poincar\'e showed the study of the dynamics of certain cases of the restricted $3$\--Body Problem 
may be reduced to investigating area\--preserving maps (see Le Calvez~\cite{Le91} and
Mather~\cite{Ma1986} for an introduction to area\--preserving maps). 
He concluded that there is no reasonable way to solve the problem {explicitly} in
the sense of finding formulae for the trajectories.  New insights appear regularly 
(eg.~ Albers et al.~\cite{Albers}, Bruno~\cite{Bruno},  
Galante et al.~\cite{GaKa2011}, and Weinstein~\cite{Weinstein1983a}).   
Instead of aiming at finding the trajectories, in dynamical systems one aims at 
describing their analytical and topological behavior.  
Of a particular interest are the constant ones, i.e., the fixed points.   

The development of  the modern field of dynamical systems was markedly influenced  
by Poincar\'e's work in mechanics, which led him to state (1912)  the Poincar\'e\--Birkhoff 
Theorem \cite{Po12, Bi1}. It was proved in full by Birkhoff in 1925. 
The result says that {an area\--preserving periodic 
twist  map  $F \colon \mathcal{S} \to \mathcal{S}$ of $\mathcal{S}:=\mathbb{R}\times [-1,\,1]$ 
has two geometrically distinct fixed points}; see Appendix A (Section~\ref{sec:classical}) for further
explanations). For the purpose of our article, its most useful 
proof  follows Chaperon's viewpoint \cite{Ch1984,Ch1984b, Ch1989} and the 
so called theory of ``generating functions". Generalizations including a number of new ideas
have been obtained by several authors, 
eg. see Carter \cite{C},  Ding \cite{D},  Franks \cite{F1,Fr88b, F3}, Le Calvez\--Wang \cite{Leca2010}, 
Neumann \cite{Ne77}, and Jacobowitz \cite{J1, J2}. 
Arnol'd realized that its generalization to higher dimensions
concerned symplectic maps and formulated
the Arnol'd Conjecture \cite{Ar1978} (see also Hofer et. al \cite{HoZe1994} and 
Zehnder \cite{Ze1986}).  

 The theme of our article is randomness.  We explore a  parallel generalization of the Poincar\'e\--Birkhoff Theorem 
 to twist maps that are random with respect to a given
 probability measure.    As 
 we will see, the stochastic theory and results we prove in this paper \emph{include
 as particular instances the classical theory of twist maps, as well as the Poincar\'e\--Birkhoff Theorem;
 this is explained in Appendix B} (Section~\ref{appendixB}).  While random dynamics has been
explored quite throughly, eg.~Brownian motions~\cite{Eins1956, Ne1967}, the implications of the area\--preservation assumption remain relatively unknown.

\subsection{Set up}
The natural setting to study area\--preserving dynamics is 
a \emph{probability space}, that is, a quadruple:
\begin{eqnarray}
\hat{\Omega}:=(\Omega, \,\cF,\, \mathbb{P},\,\tau). \label{omega}
\end{eqnarray}
 Here $\Omega$ is 
a separable metric space, 
$\cF$ is the Borel
sigma-algebra on $\Omega$,  $\tau  \colon \mathbb{R} \times \Omega \to \Omega$ is
a continuous $\mathbb{R}$\--action, and $\mathbb{P}$ is a $\tau$\--invariant ergodic 
probability measure on $(\Omega,\cF)$. Denote $\tau_a:=\tau(a,\cdot) \colon \Omega \to \Omega$. 
In addition, we assume: 
\begin{enumerate}[(i)] 
\item \label{ps}
\emph{$\mathbb{P}$\--positivity}: if $U \in \mathcal{F}$ is a nonempty open set, then $\bP(U)>0$. 
\item \label{p}
\emph{$\mathbb{P}$\--preservation by $\tau$}: $\mathbb{P}(\tau_aA) =\mathbb{P}( A)$ for every $a \in \bR$, and every $A \in \mathcal{F}$.
\item \label{pp}
\emph{Ergodicity}:
for every $A \in \mathcal{F}$, if $\tau_aA = A$ for all $a \in \bR$, then $\mathbb{P}(A) = 1$ or $\mathbb{P}(A)=0$.
\end{enumerate}
If (\ref{ps}), (\ref{p}), and (\ref{pp}) hold we say that $\mathbb{P}$ is a $\tau$\--\emph{invariant ergodic}
probability measure. For instance,  take a smooth manifold $\Omega$ which admits a smooth 
global flow $\phi:\mathbb{R}\times\Omega\rightarrow\Omega$ with an ergodic invariant
probability measure $\mathbb{P}$ that is positive on nonempty open subsets of $\Omega$ (it is non-trivial to 
find $\phi$ with these properties), $\mathcal{F}$  the 
Borel sigma-algebra of $\Omega$, and $\tau_a:=\phi(a,\cdot)$.  

\subsection{Definitions}

In what follows, let $\hat \Omega$ be a probability space as in (\ref{omega}). 
Let   $\bar F \colon  \Omega \times [-1,\,1]\to \mathcal{S}$ be a measurable map with respect to the
product measure of $\mathbb{P}$ and the Lebesgue measure on $[-1,\,1]$.   
Write $\bar F(\o,p)=(\bar Q(\o,p),\bar P(\o,p))$ and suppose that  
$F \colon \mathcal{S} \times \Omega \to \mathcal{S}$ is of the form
\begin{eqnarray} \label{then1}
\textup{\,\,}\hspace{6mm}\textup{\,}F(q,p;\o)=(Q(q,\,p;\, \omega),P(q,\,p;\, \omega))\,{\tiny \textup{with}}
    {\small \left\{
      \begin{aligned}
       Q(q,\,p;\, \omega)&\,=\, q+\bar Q(\tau_q\omega ,\,p )\\
        P(q,\,p;\, \omega)&\,=\,\bar P(\tau_q\omega,\, p).\\
         \end{aligned}\right.}
 \end{eqnarray}
 Write  $\bE$ for the expected value with respect to the probability measure $\bP$.

\begin{definition} \label{def1}
We say that $F$ in (\ref{then1})  is 
an \emph{area\--preserving random  twist} if the following hold for $\bP$-almost all $\o$:
\begin{itemize}
\item[(1)] {\emph{area\--preservation}}: $F(\cdot\,,\, \cdot\,;\, \omega)  \colon \mathcal{S} \to \mathcal{S}$
is an area\--preserving diffeomorphism;
\item[(2)] {\emph{boundary invariance}}:  $P(q,\pm 1;\o)=\pm 1$;

\item[(3)] {\emph{boundary twisting}}:  $q\mapsto Q(q,\pm 1;\o)$ is increasing, and $\pm \bar Q(\o, \pm 1)>0$;
\item[(4)] {\emph{finite second moment}}: $\sup_p\bE\big[\bar Q^2(\o,p)+\bar P^2(\o,p)\big]<\i$.
\end{itemize}
\end{definition}

\begin{definition}
An area\--preserving random twist 
$F$  is  \emph{positive monotone} if  $f\colon [-1,\,1]\to \mathbb{R} $ given by $f(p):=\bar Q(\o,\, p)$ is 
 increasing with probability one.   A measurable  map
 $\bar G=(\bar Q,\bar P) \colon\Omega  \times [-1,\,1]  \to \mathcal{S}$
is a \emph{negative monotone} area\--preserving random twist if 
$G(q,p;\o):=(q+\bar Q(\tau_q\o,p),\bar P(\tau_q\o,p))$ is the 
inverse of a positive area\--preserving random twist.\footnote{This means that 
$g(p):=\bar Q(\omega,\, p)$ is decreasing with probability $1$ and 
that  $G$ satisfies (1), (2) and (4)  but instead of (3) we have
that $Q(q,\pm 1;\o)$ is increasing but  $\pm\bar Q(\o,\pm 1)<0$.}  We say that $F$ is \emph{monotone} if 
$F$ is either positive  or negative monotone. 
\end{definition}

\begin{definition}
 $F$  is  \emph{regular} if the derivatives of $F$ and $F^{-1}$
are uniformly bounded by a constant independent of $\o$ with probability one.  
\end{definition}

Our theorems apply to twists connected to the identity. 

\begin{definition} \label{idf}
A regular area\--preserving random twist $F \colon  \mathcal{S} \times \Omega \to \mathcal{S}$
 is \emph{isotopic to the identity} if
 there is a path $(F^t \,\, | \,\,t\in[0,1])$
of diffeomorphisms $F^t \colon \mathcal{S} \times \Omega \to \mathcal{S}$ connecting $F$ 
to the identity such that
for every  $t\in[0,1]$:
\begin{itemize}
\item[(a)]
$F^t$ is a \emph{stationary lift}, i.e. it is of the form
$(q+\bar Q^t(\tau_q\o,p),\bar P^t(\tau_q\o,p));$
\item[(b)]
we have the \emph{normalization} condition:
$\frac 12\int_{-1}^{1}\bE \det({\rm d} F^t)\,{\rm d}p=1;$ 
\item[(c)]
$F^t$ is \emph{regular}, i.e. 
$\frac {dF^t}{dt}$, $F^t$ and $(F^t)^{-1}$ are almost surely bounded in ${\rm C}^r$ for sufficiently large
$r$.
\end{itemize}
\end{definition}

\subsection{Theorems}

A fixed point $(q,\,p)$ of
$F(\cdot,\,\cdot;\,\omega) \colon \mathcal{S} \to
\mathcal{S}$ is of  \emph{positive} 
(respectively \emph{negative}) type if the eigenvalues of ${\rm D}F(q,\,p;\,\omega)$ are positive 
(respectively negative).  For a set $B$, $\#B$ denote its cardinality.

\begin{theoremL} \label{epbt}
If a regular area\--preserving random twist  map $F \colon \mathcal{S} \times \Omega \to \mathcal{S}$ is isotopic to the identity, then the probability that $F(\cdot,\cdot;\omega)$ has infinitely many fixed points is one, i.e.
  $$\mathbb{P}\Big(\# {\rm Fixed \,point\, set\, of\, }F(\cdot,\,\cdot,\,;\omega)=\infty\Big)=1.$$ Moreover, if $F$ is monotone, the probability that
$F(\cdot,\cdot;\omega)$ has infinitely many fixed points of positive type is one, and the probability that
$F(\cdot,\cdot;\omega)$ has  infinitely many of negative type is one.
\end{theoremL}

For simplicity of notation, when it is clear from the context, 
sometimes we write $F$ instead of $F(\cdot,\cdot;\omega)$, even if $\omega$ is fixed.

\begin{theoremL} \label{dec}
Let
$
F \colon  \mathcal{S} \times \Omega \to \mathcal{S}
$
be a regular area\--preserving random twist map. 
Suppose that $F$ is isotopic to the identity.
Then there exists an integer $N \geq 0$ and regular area\--preserving random twists
$F_j$, where $0 \leq j \leq N,$
such that for each fixed $\omega \in \Omega$, we have a decomposition:
\begin{equation}\label{eq4.1}
F(\cdot,\cdot;\omega)=F_N(\cdot,\cdot;\omega)\circ \ldots\circ F_2(\cdot,\cdot;\omega)\circ F_1(\cdot,\cdot;\omega)\circ F_0(\cdot,\cdot;\omega),
\end{equation}
where:
\begin{itemize}
\item 
$F_j$ is negative monotone if  $j$ is even;
\item
 $F_j$ is positive monotone if $j$ is odd.
 \end{itemize}
 \end{theoremL}
 
 The integer $N$ in  (\ref{eq4.1}) is  the \emph{complexity of} $F$.  Statements 
 \cite[Propositions 2.6 \& 2.7, Lemma 2.16]{Le91} 
have the flavor to Theorem~\ref{dec} for classical twists (see also \cite[Section~9.2]{MS98}).

\begin{theoremL}\label{epbt0}
Let $N\geq0$ be an integer and let $F_j \colon \mathcal{S} \times \Omega \to \mathcal{S}$, where $0\leq j \leq N$,  
be regular area\--preserving random monotone twists such that
$F_j$ is negative monotone if  $j$ is even, and $F_j$ is positive monotone if $j$ is odd.  Then: 
\begin{enumerate}[{\rm (1)}]
\item \label{+1}
the probability that $F_i(\cdot,\,\cdot;\, \omega)$ has  infinitely many fixed points of negative type is  one,  and 
the probability that it has infinitely many fixed points of positive type  is one; 
\item \label{+2}
the composite map 
$
F_N(\cdot,\cdot;\omega)\circ \ldots\circ F_2(\cdot,\cdot;\omega)\circ F_1(\cdot,\cdot;\omega)\circ F_0(\cdot,\cdot;\omega),
$
is an area preserving random twist and the probability that 
 $F(\cdot,\cdot;\omega) \colon \mathcal{S} \to \mathcal{S}$  has infinitely many fixed points is one.   
\end{enumerate}
\end{theoremL}

The classical Poincar\'e\--Birkhoff Theorem from 1912 Appendix A (Section~\ref{sec:classical}) is
 a particular example of our stochastic setting; we explain this by giving the concrete stochastic models
which recover the classical theory in Appendix B (Section~\ref{appendixB}).

\begin{example} \label{expr}
The following are quadruples 
$(\Omega, \,\cF,\, \mathbb{P},\,\tau)$ as in \eqref{omega}. In each case $\cF$ is
 the Borel $\sigma$-algebra associated with the natural topology on $\Omega$.
(i) Let $v \in \mathbb{R}^k$ such that $\langle v,\,   n\rangle=0$ for $n \in \mathbb{Z}^k$ implies 
$n=0$.  Let $\Omega=\mathbb{T}^k:=(S^1)^k$ and  
$
\tau_a\omega:=\omega+av \,(\!\!\!\!\!\!\,\, \mod 1),
$
where $S^1$ is $[0,1]$ with
$0$ and $1$ identified.
Let $\mathbb{P}$ be the normalized Lebesgue (Haar) measure.
(ii)
Let $\Omega$ be the set of discrete infinite subsets of $\mathbb{R}$. Every 
$\omega \in \Omega$ may be written as
$\omega=\{x_i \,\, | \,\, i \in \mathbb{Z}\} \subset \mathbb{R},$ and we define
$
\tau_a(\omega):=\{x_i+a\,\, | \,\, i \in \mathbb{Z}\}. 
$
Let $\mathbb{P}$ be a Poisson random measure of intensity $1$.  
\end{example}

\begin{remark}
\normalfont
 If $F$ is a flow map of a Hamiltonian system associated with a smooth
Hamiltonian function of compact support, then it is isotopic to identity and the condition $(b)$ of Definition~\ref{idf}
is trivially satisfied because $\det {\rm D}F=1$. We refer to part {\bf{(v)}} of Section~\ref{appendixB} for more details. As we will see in the process of proving Theorem~\ref{dec}, if $F$ is isotopic to identity and the path $(F^t:\ t\in[0,1])$ may be deformed to a new path that is the flow of a stationary Hamiltonian system.
\end{remark}

\begin{remark}
\normalfont
The first breakthrough on A'rnold's Conjecture was done by Conley and Zehnder~\cite{CoZe1983}. According to
their theorem, any smooth symplectic map $F:\bT^{2d}\to\bT^{2d}$ that is isotopic to identity has at least $2d+1$ many fixed points. For the stochastic analog of \cite{CoZe1983}, we take a $2d$-dimensional stationary 
process $X(x;\o)=\bar X(\tau_x\o)$ with $\bar X:\O\to\bR^{2d};\ x\in\bR^{2d},$ and assume that 
its lift $F(x;\o)=x+X(x;\o)$ is symplectic with probability one. Our strategy of proof is also applicable to such 
random symplectic maps. The main ingredients for proving results analogous to Theorems~\ref{epbt}-\ref{epbt0} are 
 Morse Theory and Spectral Theorem for  multi-dimensional stationary processes. In a subsequent paper,
we will work out a generalization of Conley and Zehnder's Theorem in the stochastic setting.
\end{remark}

\begin{figure}[H]
  \centering
{\tiny
  \begin{tabular}{l|cccc}
  Complexity & $N=0$ & $N=1$ & $N=2$ & $N \geq 3$\\
    \hline
    {Existence} 
    &  Prop.~\ref{prop:key} &  Thm.~\ref{keyresult}(\ref{xa}) & Thm.~\ref{keyresult2}(\ref{xxa}) & Thm.~\ref{keyresult3}(\ref{xxxb})\\
   {Additional} 
   & Thms.~\ref{prop9.2} \& \ref{prop9.5} & Thms.~\ref{keyresult}(\ref{xb}) \& \ref{prop9.7} & Thm.~\ref{keyresult2}(\ref{xxb})&  Thm.~\ref{keyresult3}(\ref{xxxa})
     \end{tabular}}
  \caption{{\small Depending on $N$ the 
proof of 
Theorem~\ref{epbt0} (\ref{+2}) is different.  {Positivity/negativity  in Theorem~\ref{epbt} follow from Theorem~\ref{prop9.2};
Theorem~\ref{prop9.2} constructs monotone twists; 
Theorem~\ref{prop9.5} describes the ``density'' (Definition~\ref{density}) and ``spectral" nature of 
fixed points.  Theorem~\ref{prop9.7} does the analogue if $N=1$.   Theorems~\ref{keyresult}(\ref{xb}) \& \ref{keyresult2} (\ref{xxb}), and \ref{keyresult3}(\ref{xxxa}),
describe further the fixed points. }}}
  \label{fig:comparison}
\end{figure}
Poincar\'e understood that preserving area has global implications for a dynamical
system. We give instances when this connection persists in a random setting. We do it
by using random generating functions to reduce the proofs to
finding critical points of random maps. In Section~\ref{ggfsec}
we define them, and explain how to use them to show the main results.  Section~\ref{sec7}  proves 
Theorem~\ref{dec}. The  sections which follow
contain a case\--by\--case proof ($N=0$, $N=1$, $N=2$, $N\geq 3$) 
of Theorem~\ref{epbt0}. For $N=0,\,1$ we have additional results. Section~\ref{sec:classical}  reviews the classical theory.
We recommend \cite{ArAv1968, Ka, Ko1957, Mo1973, Sm1967}  for modern accounts of dynamics, and 
\cite{BrHo2012, HoZe1994, MS98, Pol01} for treatments emphasizing symplectic techniques.

\section{\textcolor{black}{Calculus of random generating functions}} \label{ggfsec}

We construct the principal novelty of the paper, random generating functions, and explain how to use them
to find fixed points.  Recall that $\Omega$ is  as in (\ref{omega}).

\begin{definition}
 We say that a measurable function $G:\O\to\bR$ is \emph{$\o$-differentiable} if the limit
$
\nabla G(\o):=\lim_{t \downarrow 0} t^{-1} \,(G( \tau_t\omega) -G( \omega))
$
exists for $\bP$\--almost all $\o$.  For a measurable map 
$K \colon \Omega\times  [0,\,1] \to \mathbb{R}$ we write  $K_p=\frac{\partial K}{\partial p}$  and
 $K_\o=\frac{\partial K}{\partial \omega}$ for the partial derivatives of $K$. We say that 
 \emph{$K$ is ${\rm C}^1$}  if the partial derivatives  
of $K$ exist and are continuous for $\bP$\--almost all $\o$.
 \end{definition}

Given an area\--preserving random twist as in \eqref{then1}, 
 consider the sets (see Figure~\ref{Lfig2}):
\begin{eqnarray} \label{abar}
\left\{ \begin{array}{rl} \label{abar}
\bar{A}:=\{(\omega;\,v)\,\, | \,\, \bar Q(\omega,\,-1 ) \leq v \leq \bar Q(\omega,\,1 )\}  \subseteq 
 \Omega\times \mathbb{R} \\
\bar{A}_\o:=\{v\,\, | \,\, (\omega;\,v)\in \bar{A}\} \subseteq \mathbb{R}\\
 \bar{A}^v:=\{\o\,\, | \,\, (\omega;\,v)\in \bar{A}\}\subseteq \Omega \\
{A}_\o:=\{(q,Q)\,\, | \,\, (\tau_q\omega;\,Q-q)\in \bar{A}\} .
          \end{array} \right . 
\end{eqnarray}
We  write 
$F^{-1}(P,Q)=(q(Q,P),p(Q,P))$.

\begin{figure}[htbp]
\begin{center}
\psfrag{fudge}{$Q$}
\psfrag{yy}{$q \mapsto Q(q,\,-1)$}
\psfrag{qq}{$q$}
\psfrag{candy}{$q \mapsto Q(q,\,1)$}
\includegraphics[width=2.1in]{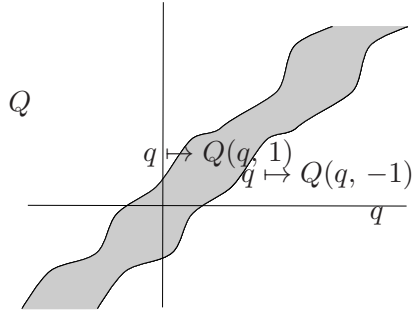}
\caption{{\small ${A}_\o$  in \eqref{abar} bounded by
  the graphs of $q \mapsto Q(q,\,1)$, $q \mapsto Q(q,\,-1)$, respectively.}}
\label{Lfig2}
\end{center}
\end{figure}

\begin{definition}
Given an area\--preserving random twist map  \eqref{then1}, we say  that
$\cL:{\bar A}\times \bR^N\to\bR$
 is a {\em{generalized generating 
function  of complexity  $N$}}  if $\cL$ is $\op{C}^1$ and the
function 
$
\cG(q,Q;\xi)=\cG(q,Q;\xi,\o):=\cL(\tau_q\o,Q-q,\xi_1-q,\dots,\xi_N-q),
$
with, $\xi=(\xi_1,\dots,\xi_N)$,  satisfies:
\begin{equation} \label{key:for}
\cG_\xi(q,Q;\xi,\o)=0\Rightarrow F\left(q,-\cG_q(q,Q;\xi,\o);\o\right)=\left(Q,\cG_Q(q,Q;\xi,\o)\right).
\end{equation}
\end{definition}

Our interest in generalized generating functions is due to the following.

\begin{prop} \label{easy:lem3}
Let $\cL$ be a generalized generating function for $F$. 
Set 
$$\cI(q,\xi;\o)=\cL(\tau_q\o,0,\xi_1-q,\dots,\xi_N-q).$$ 
If $(\bar q,\bar\xi)$  is a critical point for $\cI(\cdot,\cdot;\o)$, then 
$\vec{x}:=(\bar q,\, -\mathcal{G}_q(\bar q,\,\bar q;\, \bar\xi))$ is a 
fixed point of $F(\cdot,\cdot;\o)$.
\end{prop}

\begin{proof} Observe that if $(\bar q,\bar\xi)$ is a critical point of $\cI$, then by the definition of $\cG$,
$
 \mathcal{G}_Q(\bar q,\,\bar q;\, \bar\xi)=- \cG_q(\bar q,\,\bar q;\,\bar \xi)$ and
$\cG_\xi(\bar q,\, \bar q;\, \bar\xi)=0$. Since $\cL$ is a generating function, $\cG_\xi=0$ gives $F(\vec x)=\vec x.$ 
\end{proof}

The strategy to prove Theorem \ref{epbt0} is to show that fixed points of $F$ are in  correspondence with critical points of the associated {random} generating function $\mathcal{G}$, and 
then prove existence of critical points of $\mathcal{G}$.    
Viterbo  has used generating functions with great success~\cite{Vi2011}.  
 Gol\'e~\cite{Go2001}  describes several results in this direction.

\section{Proof of Theorem \ref{dec}} \label{sec7}

We begin by introducing  stationary lifts.

\begin{definition}
A function $f(q,\omega)$ is \emph{stationary} if  $f(q,\omega)=\bar f (\tau_q\omega)$  
for a continuous
$\bar f \colon \Omega \to \mathbb{R}$. We say that $f$ is a \emph{stationary lift} if
$
f(q,\omega)=q+\bar f (\tau_q\omega)
$
for a continuous $\bar f \colon \Omega \to \mathbb{R}$. 
\end{definition}

\begin{definition}
A vector\--valued
map $f(q,p;\omega)$ with $f( \cdot,\omega) \colon \mathbb{R}^2 \to \mathbb{R}^2$
is \emph{$q$\--stationary} if 
$f(q,p,\omega)=\bar f (\tau_q\omega,p)$
for some $\bar f \colon \Omega \times \mathbb{R} \to \mathbb{R}^2$. A similar
definition is given for $f(\cdot,\omega) \colon \mathcal{S} \to \mathbb{R}^2$.
We say that such $f$ is a \emph{$q$\--stationary lift} if $f$ can be expressed as
$
f(q,p,\omega)=(q,0)+\bar f (\tau_q\omega,p).
$
\end{definition}

\begin{prop} \label{prop:lifts}
The following properties hold:
\begin{enumerate}[{\rm (P.1)}]
\item \label{item:1}
If $f(q,\omega)$ is an increasing stationary lift in ${\rm C}^1$, then $f^{-1}$ is 
an increasing lift. The same holds for $q$\--stationary diffeomorphism lifts $f(q,p,\o)$.
\item \label{item:2}
The composition of $q$\--stationary lifts is a $q$\--stationary lift.
If $f$ is a $q$\--stationary lift and $g$ is $q$\--stationary, $g \circ f$ is $q$\--stationary.
\item \label{item:3}
For every differentiable $\bar f:\O\to\bR$ we have that $\bE \nabla \bar f=0.$
\end{enumerate}
\end{prop}
\begin{proof}
The proof of (P.\ref{item:2}) is trivial. 
 We only prove (P.\ref{item:1}) for a stationary lift $f(q,p,\o)$ because the case of $f(q,\o)$ is done 
 in the same way.  Assume that $f(q,p,\omega)$ is a $q$\--stationary lift so that for every $a\in\bR$,
$
f(q+a,p,\omega)=(a,0)+f(q,p,\tau_a\omega),
$
and write
$g(q,p,\omega)$ for its inverse. To show that $g(q,p,\omega)$ is a $q$\--stationary
lift it suffices to check that
$
g(q+a,p,\omega)=(a,0)+g(q,p,\tau_a\omega).
$
In order to do this, let us fix $a$ and write $\tilde{g} (q,p,\omega)$ for the right-hand side
$(a,0)+g(q,p,\tau_a\omega)$. Observe that 
since $f$ is a $q$\--stationary lift,
$
f(\tilde{g}(q,p,\omega),\omega)=(a,0)+f(g(q,p,\tau_a\omega),\tau_a\omega)  
=(a,0)+(q,p)=(q+a,p)$. By uniqueness, 
$
\tilde{g}(q,p,\omega)=g(q+a,p,\omega),
$
which concludes the proof of (P.\ref{item:2}). As for (P.\ref{item:3}), write $f(x,\o)=\bar f(\tau_x\o)$ and
observe that for any smooth $J \colon \bR\to\bR$ of compact support, with $\int_{\bR} J(x){\rm d}x=1$,
\begin{eqnarray}
\bE \nabla \bar f&=&\int_{\bR} J(x)\left(\bE   f_x(x,\o)\right)\ {\rm d}x \nonumber\\ 
&=&-\bE\int_{\bR} J'(x) f(x,\o)\ {\rm d}x  \nonumber \\
&=&-\left(\int_{\bR} J'(x)\ {\rm d}x\right)\left(\bE\bar f\right),
\end{eqnarray}
so $\bE \nabla \bar f=0$.  
\end{proof}

The proof of Theorem \ref{dec} draws on spectral  theory 
for random processes. To this end, let us recall the statement of the Spectral Theorem for random processes. The Spectral Theorem allows us to represent a random process in terms of an auxiliary process with randomly orthogonal increments. Such a representation reduces to a Fourier series expansion if the stationary process is periodic. In order to
 apply the Spectral Theorem to a stationary process $a(q)=\bar a(\tau_q\o)$,
one  follows the steps:
 \begin{itemize}
 \item[(i)]
 Assume that $a(q)$ is {\it {centered}} in the sense that $\mathbb{E}\bar a(\omega)=0$.
 We  define the {\it{correlation}}
 $
 R(z)=\mathbb{E}\bar a(\omega)\bar a(\tau_z\omega).
 $
 \item[(ii)]
There always exists a nonnegative measure $G$ such that
 $
 R(z)=\int_{-\infty}^{\infty} {\rm e}^{{\rm i}z \xi} \, G({\rm d}\xi).
 $
 \item[(iii)]
 One can construct an auxiliary  process $(Y(\xi):\xi\in\bR)$ or alternatively the random measure 
$Y({\rm d}\xi)=Y({\rm d}\xi,\o)$ 
that are related by
 $Y(I)=Y(b)-Y(a)$, where $I=[a,b]$.
 The process $Y$ has orthogonal increments in the following sense:
 \begin{eqnarray}
 I \cap J = \varnothing  \Longrightarrow \mathbb{E} Y(I)Y(J)=0 .\label{ortho}
 \end{eqnarray}
 The relationship between  the measure $G({\rm d}\xi)$  or its associated nondecreasing function $G(\xi)$ is given by
 $\mathbb{E}Y(I)^2=G(b)-G(a)=G(I)$. \end{itemize}
 The \emph{Spectral Theorem} (\cite{Doob})  
 says that  for any stationary process $a$ for which $\bE \bar a^2<\i$,
we may find a process $Y$ satisfying \eqref{ortho}  such that
 $
 \bar a(\tau_q\omega)=\int_{-\infty}^{\infty} {\rm e}^{{\rm i}q  \xi} Y({\rm d}\xi).
 $
 Note that 
 \begin{eqnarray}
 \mathbb{E}\bar a(\tau_q\omega)\bar a(\omega)
 =\mathbb{E}\int_{-\infty}^{\infty} {\rm e}^{{\rm i}q  \xi} \, Y({\rm d}\xi)
 \int_{-\infty}^{\infty}Y({\rm d}\xi') 
 = \int_{-\infty}^{\infty} {\rm e}^{{\rm i}q \xi} \, G({\rm d}\xi).
 \end{eqnarray}
 Also, $\bar a(\omega)=\int_{-\infty}^{\infty}Y({\rm d}\xi,\o),$ 
and the stationarity of $a(q,\o)$ means 
$$
 Y({\rm d}\xi,\tau_q\o)={\rm e}^{{\rm i}q  \xi}Y({\rm d}\xi,\o).
$$
For our application below, we will have a family of random maps $(a(q,t)\,\, | \,\, t\in[0,1])$ that varies smoothly with $t$. In this case we can guarantee that the associated measures $Y({\rm d}\xi,t)$ 
depend smoothly in $t$.

The main difficulties of the proof are due to the fact that the ``random and area\--preservation properties" do not integrate 
well, for instance when arguing about $t$\--dependent deformations 
which must preserve both properties. The proof  consists of four steps.
\\
\\
\emph{Proof of Theorem \ref{dec}}.  Write $x=(q,p)$. Since $F$ is random isotopic to the identity, there is  a path 
$
\cF=(F^t \,\, | \,\, t\in[0,1])
$
of diffeomorphisms that connects $F$ to the identity map, $F^t$ is a stationary lift for each $t \in [0,1]$,  we have the normalization
$
\frac 12\int_{-1}^{1}\bE \det({\rm d} F^t){\rm d}p=1
$
for every $t\in[0,1]$, and $F^t$ is regular for a constant independent of $t$. 
There are four steps to the proof:

 \paragraph{\emph{Step} 1} (\emph{General strategy to turn $\mathcal{F}$ into a path of area\--preserving random twists}).  Write
 $
 \rho^t(x)=\rho^t(q,p)=\bar\rho^t(\tau_q\o,p)={\rm det}({\rm d}F^t(x)),
 $
 so that $(F^t)^*\ {\rm d}x=\rho^t\ {\rm d}x$, where ${\rm d}x={\rm d}q \wedge {\rm d}p$,
 and by assumption,
$$
\frac 12\int_{-1}^{1}\bE \rho^t{\rm d}p=\frac 12\int_{-1}^{1}\bE \bar\rho^t{\rm d}p=1, \,\,\,\rho^0=\rho^1=1.$$ 
Since $F^t$ is regular uniformly on $t$, the function $\rho^t$ is bounded and bounded away from $0$
by a constant that is independent of $t$. That is,  there exists a constant $C_0>0$ such that 
$
C_0^{-1}\leq \rho^t(x;\o)\leq C_0,
$
 almost surely.
We now construct, out of $F^t$, an area\--preserving path $\Lambda^t$ which is a stationary lift for every $t$.
We achieve this by using Moser's deformation trick, namely 
 we construct a path $G^t$ such that $\L^t=F^t \circ G^t$ is an
 area\--preserving stationary lift for all $t$. As it will be clear from the construction of $G^t$ below,
$G^0$ and $G^1$ are both the identity and, as a result, $\L^t$ is a path of area\--preserving maps that connects $F$ to identity.
We need $(G^t)^*(\rho^t {\rm d}x)={\rm d}x,$ and $G^t$ is constructed
as a $1$-flow map of a vector field $\cX(x,\theta)=\cX(x,\theta;t)$.
  So  we wish to find some vector field 
  $\cX$ such that $G^t=\phi^1$ where $\phi^{\theta}$, $\theta \in [0,1]$, denotes the flow
of  $\cX$. In fact, we also have to make sure that  the vector field $\pm \cX$ is parallel to the $q$-axis
 at $p=\pm 1$. This  guarantees that the strip $\cS$ is invariant under the flow of $\cX$. 
 
 Let 
 $
 m(\theta,x):=\theta\rho^t(x)+(1-\theta),
 $
 so that $m(\th,x)\ {\rm d}x$ is connecting the area form ${\rm d}x$ to $\rho^t \ {\rm d}x$. We need to 
 find a vector field $\cX$ such that its flow $\phi^\th$ satisfies 
$
(\phi^\th)^*{\rm d}x=m(\th,x)\ {\rm d}x.
$
Equivalently, $m$ must satisfies the Liouville's  equation
 \begin{eqnarray} \label{equation}
 m_\th+\nabla \cdot (\mathcal{X} m)=\rho^t-1+\nabla \cdot (\mathcal{X} m)=0.
 \end{eqnarray}
The strategy to solve equation (\ref{equation}) for $\cX$  is as follows. Search for a solution $\cX$ such that
$m\cX=\nabla_xu$ is a gradient. Of course we insist that $u$ is $q$\--stationary so that $\cX$ is also
$q$-stationary; 
\begin{eqnarray}
u(q,p,\th)&=&\bar u(\tau_q\o,p,\th), \nonumber\\
 (m\cX)(q,p,\th)&=&(m\cX)(q,p,\th;\o)=(\bar m\bar\cX)(\tau_q\o,p,\th) 
=(\bar u_\o(\tau_q\o,p,\th), \bar u_p(\tau_q\o,p,\th)). \nonumber
\end{eqnarray}
  Since $t$ is fixed, we  drop $t$ from our notations and write $\rho^t=\rho$. 
The equation (\ref{equation}) in terms of $u$ is an elliptic partial differential equation of the form
 \begin{eqnarray} \label{equation2}
\D u=1-\rho=: \eta,
 \end{eqnarray}
with $\eta(q,p)=\bar\eta(\tau_q\o,p)$ and $\int_{-1}^{1}\bE \eta(\o,p){\rm d}p=0.$
This concludes Step 1.

 \paragraph{\emph{Step} 2}  (\emph{Applying Spectral Theorem to solve} (\ref{equation2})).   
To apply the Spectral Theorem for each $p$, set $\hat \eta(\o,p)=\bar\eta(\o,p)-k(p)$ for 
$k(p)=\bE \bar\eta(\o,p)$, and write
  \[
  R(q;p):=\mathbb{E}\hat \eta(\omega,p)\hat \eta(\tau_q\omega,p)
=\int_{-\infty}^{\infty}   {\rm e}^{{\rm i}qz} \, G({\rm d}z,p).
\]
Note that $\bE \hat \eta(\o,p)=0$ for every $p$ and $\int_{-1}^1k(p){\rm d}p=0.$
 We have the representation
 \begin{eqnarray} \label{eq:m}
 \eta(q,p)=k(p)+\bar \eta(\tau_q\omega,p)=k(p)+\int_{-\infty}^{\infty} {\rm e}^{{\rm i}qz} \, Y({\rm d}z,p),
 \end{eqnarray}
 where $Y({\rm d}z,p)=Y({\rm d}z,p;\o)$ satisfies 
\begin{equation}\label{eq3.6}
Y({\rm d}z,p;\tau_q\o)={\rm e}^{{\rm i}qz}Y({\rm d}z,p;\o). 
\end{equation}
We want to find a solution to the partial differential equation
$\Delta u(q,p)=\eta(q,p)$, which is still stationary in the $q$ variable. First choose $h_0(p)$ such that $h_0''(p)=k(p)$
and satisfy the boundary conditions 
\begin{eqnarray} \label{bc}
h_0(\pm 1)=0. 
\end{eqnarray}
We 
write $u=h_0+v$ and search for a random $v$ satisfying
 \begin{equation} \label{eq:star2}
 \Delta v(q,p)=\hat \eta(q,p):=\hat\eta(\tau_q\o,p). \nonumber 
 \end{equation}
 Since $\g(q,p)={\rm e}^{({\rm i}q\pm p)z}$ is harmonic for each $z\in\bR$, the function $h$ given by
 \begin{eqnarray} \label{eq:c1}
h(q,p):=\int_{-\infty}^{\infty}{\rm e}^{{\rm i}qz}\, 
\big( {\rm e}^{zp}\, \Gamma_1({\rm d}z)+ {\rm e}^{-zp}\, \Gamma_2({\rm d}z)   \big),
 \end{eqnarray}
is harmonic for any measures $\G_1$ and $\G_2$. We will find a solution of the form $v=w+h$ where
$\Delta w={\eta}$ and $h$ will be selected to satisfy the boundary conditions $v_p(q,\pm 1)=0$. 
Indeed $w$  given by
  \begin{eqnarray} 
 w(q,p)&:=& \,\, \int_{-1}^p \int_{-\infty}^{\infty} \frac{{\rm e}^{{\rm i}qz}}{z} 
\sinh ((p-a)z) Y({\rm d}z,a) {\rm d}a   \nonumber \\ 
 &=&\int_{-1}^p \int_{-\infty}^{\infty}  
{\rm e}^{{\rm i}qz} \, \frac{{\rm e}^{(p-a)z} - {\rm e}^{(a-p)z}}{2z} Y({\rm d}{z},a) {\rm d}a,\nonumber
  \end{eqnarray}
satisfies all of the required properties. In order to verify this observe that
 \begin{align*}
w_{qq}(q,p)&=-\frac 12 \int_{-1}^p\int_{-\infty}^{\infty} z{{\rm e}^{{\rm i}qz}} 
 \left( {{\rm e}^{(p-a)z} - {\rm e}^{(a-p)z}}\right)Y({\rm d}z,a) {\rm d}a,\\
w_p(q,p)&=\frac 12\int_{-1}^p\int_{-\infty}^{\infty}  {\rm e}^{{\rm i}qz} \,
\left( {{\rm e}^{(p-a)z} + {\rm e}^{(a-p)z}}\right) Y({\rm d}{z},a) {\rm d}a,\\
w_{pp}(q,p)&=\frac 12\int_{-1}^p\int_{-\infty}^{\infty}  z{\rm e}^{{\rm i}qz} \,
\left( {{\rm e}^{(p-a)z} - {\rm e}^{(a-p)z}}\right) Y({\rm d}{z},a) da+\hat {\eta}(q,p).
\end{align*}
This clearly implies that $\Delta w={\eta}$.

 On the other hand, the process $w$ is $q$\--stationary. In other words 
$
w(q,p)=w(q,p;\o)=\bar w(\tau_q\omega,p),
$
 for a process $\bar w$. This can be verified by checking that $w(q+b,p;\o)=w(q,p;\tau_b\o),$ which is an immediate consequence of (\ref{eq3.6}):
\begin{eqnarray}
 w(q+b,p;\o)
 = \int_{-\infty}^{\infty} \int_{-1}^p \frac{{\rm e}^{{\rm i}qz}}{z} 
 \sinh ((p-a)z) {\rm e}^{{\rm i}bz}Y({\rm d}z,a;\o) {\rm d}a  
 =w(q,p;\tau_b\o). \nonumber
\end{eqnarray}
This concludes Step 2.

\paragraph{\emph{Step} 3} (\emph{Checking that $\Gamma_1$ and $\Gamma_2$ in {\rm (\ref{eq:c1})} 
 can be chosen to satisfy the boundary conditions \eqref{bc}}). At $p=\pm 1$, $\pm\nabla u$ should
 point in the direction of the $q$\--axis, we need to have that
 $$
  u_p(q,\pm 1)=v_p(q,\pm 1)=0,
 $$
because $h_0(\pm 1)=0$. First, the condition $v_p(q,1)=0$, means
\begin{eqnarray}
\int_{-\infty}^{\infty} {\rm e}^{{\rm i}qz} z ({\rm e}^{ z} \Gamma_1({\rm d}z) 
- {\rm e}^{- z}\Gamma_2({\rm d}z))   \,\,\,\,\,\,\,\,\,\,\,\,\,\,\,\, \,\,\,\,\,\,\,\,\,\,\,\,\,\,\,\,
\end{eqnarray}
\begin{eqnarray}
\,\,\,\,\,\,\,\,\,\,\,\,\,\,\,\, \,\,\,\,\,\,\,\,\,\,\,\,\,\,\,\,+\,\frac 12\int_{-1}^1\int_{-\infty}^{\infty}{\rm e}^{{\rm i}qz}({\rm e}^{(1-a)z}+{\rm e}^{(a-1)z}) Y({\rm d}z,a){\rm d}a =0,
\nonumber 
\end{eqnarray}
 and the condition $v_p(q,-1)=0$, means
 $
\int_{-\infty}^{\infty} {\rm e}^{{\rm i}qz} z ({\rm e}^{-z} \Gamma_1({\rm d}z) 
- {\rm e}^{z}\Gamma_2({\rm d}z))=0.
$
 Since we need to verify the above conditions for all $q$,
we must have that $\Gamma_1={\rm e}^{2z}\G_2,$  and 
 $
 z{\rm e}^z({\rm e}^{2z}-{\rm e}^{-2z})\Gamma_2({\rm d}z)+Y'({\rm d}z)\,=\,0,
$
   where
 $
Y'({\rm d}z)=\frac 12 \int_{-1}^1 ({\rm e}^{(1-a)z}+{\rm e}^{(a-1)z})Y({\rm d}z,a){\rm d}a.
 $
 In summary,
 \begin{equation}\label{eq:thegammas}
  \Gamma_2({\rm d}z)=-z^{-1}{\rm e}^{-z}({\rm e}^{2z}-{\rm e}^{-2z})^{-1}Y'({\rm d}z),\ \ \ \Gamma_1
  ={\rm e}^{2z}\G_2.
\end{equation}   
 Since $Y$ satisfies (\ref{eq3.6}), the same property holds true for both $\G_1$ and $\G_2$. From this 
 it follows  that the process $h$ (and hence $u$) is $q$\--stationary; this is proven in the same way we established the stationarity of $w$. The $q$-stationarity of $u$ implies that $\cX$ is $q$\--stationary. This in turn implies that the flow $\phi^\th$ is a $q$\--stationary lift for each $\th$. To see this, observe that since both
$\phi^{\th}(q+a,p;\o)$ and $(a,0)+\phi^{\th}(q,p;\tau_a\o)$ satisfy the ordinary differential equation
  $y'(\th)=\mathcal{X}(y(\th),\th;\omega)$
 for the same initial data $(q+a,p)$, we deduce 
$
 \phi^{\th}(q+a,p;\o)= (a,0)+\phi^{\th}(q,p;\tau_a\o),
$
which concludes this step.

\paragraph{\emph{Step} 4} (\emph{Producing a twist decomposition for $F$ from the path $\L$}). 
We claim that there exists  a $q$\--stationary process
$
H(q,p,t;\omega)=\bar H (\tau_q\omega,p,t)
$
such that 
$$
\frac{{\rm d}  \L^t}{{\rm d}t}=J\, \nabla H \circ \L^t
$$
holds. Indeed, since $\L^t$ is a $q$\--stationary lift, 
$\frac{{\rm d} \L^t}{{\rm d}t}$ is $q$\--stationary. Hence
by Proposition \ref{prop:lifts}, the composite
$
\frac{{\rm d}}{{\rm d}t} \L^t \circ (\L^t)^{-1}
$
is $q$\--stationary. Set
$$
A(t,q,p;\omega)=\frac{{\rm d}\L^t }{{\rm d}t} \circ(\L^t)^{-1}(q,p,\omega).
$$
We need to express $A$ as $J \,\nabla H$. Observe that since $\L^t$ is area
preserving, $A$ is divergence free. Write
$
A(t,q,p;\omega)=(a(\tau_q\omega,p,t),b(\tau_q\omega,p,t)).
$
We have $a_{\omega}+b_p=0.$
Set
$$
H(q,p,t;\omega)=\int_0^p a(\tau_q\omega,p',t) {\rm d} p'-b(\tau_q\omega,0,t).
$$
Clearly $H_q=-b$, $H_p=a$, and $H$ is stationary. Note that since 
$ \frac{{\rm d}\L^t }{{\rm d}t}$ and $(\L^t)^{-1}$
are bounded in ${\rm C}^1$,  $A$ is bounded in ${\rm C}^1$. Let us write $(\L^{s,t}\,\, | \,\, s\le t)$ for the flow of the vector field $A$ so that $\L^{0,t}=\L^t$ and $\L^{s,s}={\rm id}$.
On the other hand
$
\frac{{\rm d}}{{\rm d}t} \L^{s,t}=A \circ \L^{s,t},
$
implies that 
$$
\frac{{\rm d}}{{\rm d}t} {\rm D}\L^{s,t}={\rm D} A \circ \L^{s,t}\ {\rm D}\L^{s,t}.
$$
Hence there are constants $c_0,\,c_1$ such that
$
 \|{\rm D}\L^{s,t}\|\leq {\rm e}^{c_0(t-s)}$ and
$\|{\rm D}\L^{s,t}-{\rm id}\|\leq c_1(t-s)$.
It follows that
$
\|  \L^{s,t}-{\rm id}     \|_{{\rm C}^1} \leq c_2 \, (t-s),
$
for a constant $c_2$. So we may write
$
F=\L^1= \psi^1 \circ \psi^2 \circ \ldots \circ \psi^n \,\,\, \, \,\,\,\, \textup{with}\,\,\,\,\,\,
\psi^j=\L^{\frac{jt}{n},\frac{(j-1)t}{n}}
$
satisfying
$
\| \psi^j-{\rm id} \|  \leq  c_2 \,n^{-1}.
$
Hence, for large $n$, we can arrange
$
\max_{1 \leq j \leq n} \| \psi^j-{\rm id}\|_{{\rm C}^1} \, \leq \, \delta.
$
Let $\varphi^0(q,p)=(q+p,p)$. Then
$$
\| \psi^j \circ \varphi^0-\varphi^0\|_{{\rm C}^1} \leq \delta.
$$
The map $\varphi^0$ is a positive monotone twist map and we can readily show that 
$\psi^j \circ \varphi^0$ is positive monotone twist if $\d<1$. Hence 
$
\psi^j=\eta^j \circ (\varphi^0)^{-1}
$ 
where $\eta^j$
is a positive monotone twist and $(\varphi^0)^{-1}$ is a negative monotone twist. This concludes the proof of
Theorem~\ref{dec}.
\qed
  
 Next we give an application to random generating functions of complexity $N$. 
 For the following, recall the definition of
$
\bar {A} 
$
in {\rm (\ref{abar})}. 

\begin{lemma} \label{genf:def}
Let $F$ be a area\--preserving random twist map  of the form
$F=F_N\circ \ldots\circ F_0$, where each $F_i$ is a monotone area\--preserving random twist
with generating function of the form
$\cG^i(q,Q;\o):=\cL^i(\tau_q\o,Q-q).$ Then $F$ has a generalized generating function 
$
\mathcal{L} \colon \bar{A} \times \mathbb{R}^{N} \to \mathbb{R}
$
of complexity $N$, $\mathcal{L}(\o,v; \, \xi)$, that is given by
$
\mathcal{L}^0(\o,\, \xi_1)+ \sum_{j=1}^{N-1}
\mathcal{L}^j (\tau_{\xi_{j}}\o,\, \xi_{j+1}-\xi_{j})+
\mathcal{L}^N(\tau_{\xi_{N}}\o,\,v-\xi_{N})
$,
or equivalently
$
\mathcal{G}(q,\,Q; \, \xi)=\mathcal{G}^0(q,\, \xi_1)+ \sum_{j=1}^{N-1}\mathcal{G}^j (\xi_{j},\, \xi_{j+1})+
\mathcal{G}^N(\xi_{N},\,Q)$ 
where $\xi=(\xi_1,\ldots,\xi_{N}) \in \mathbb{R}^{N}$.
\end{lemma}
\begin{proof} We write $\xi_0=q, \xi_{N+1}=Q$.
To verify \eqref{key:for}, observe that $\cG_\xi=0$ means that $\cG_Q^i(\xi_i,\xi_{i+1})=-\cG_q^{i+1}(\xi_{i+1},
\xi_{i+2})$ for $i=0,\dots, N-1$.
We have that
$F_i(q_i,\,p_i)=(Q_i(q_i,\,p_i),\,P_i(q_i,\,p_i))$,
with $ \mathcal{G}^i_Q(q_i,\,Q_i)=P_i$, $\mathcal{G}^i_q(q_i,\,Q_i)=-p_i$. By definition we have that
$F_0(q_1,-\mathcal{G}^0_q(q,\, \xi_1))=(\xi_1,\, 
\mathcal{G}^0_Q(q,\, \xi_1)).$
Since 
$ \mathcal{G}^0_Q(q,\, \xi_1)=- \mathcal{G}^1_q(\xi_1,\, \xi_2)$ 
we have  that
$F_1(\xi_1,\, -\mathcal{G}^1_q(\xi_1,\,\xi_2))=(\xi_2,\, 
 \mathcal{G}^1_Q(\xi_1,\, \xi_2)).$
Iterating
 $N$ times we get
$$F_N(\xi_{N},\, - \mathcal{G}^N_q(\xi_{N},\,Q))=
(Q,\,  \mathcal{G}^N_Q(\xi_{N},\,Q)), 
$$
so
$
F(q,\, - \mathcal{G}_q(q,\,Q;\, \xi))=(Q, \, \mathcal{G}_Q(q,\,Q;\, \xi)).
$
\end{proof}

\section{\textcolor{black}{Area\--preserving random monotone twists}}

This section proves a result which implies the $N=0$ case in Theorem~\ref{epbt0} (item (\ref{+1})). We also provide
complementary results on the density and spectral theoretic properties of the fixed points, and give
a method to construct monotone twists from a given smooth map.

\subsection{\textcolor{black}{Existence of random generating functions} } 

The map $v\mapsto \bar p(\o,v)$ is defined to be the inverse of the map $p\mapsto \bar Q(\o,p)$. This means that
$Q\mapsto p(q,Q)=\bar p(\tau_q\o,Q-q)$ is the inverse of  
$p\mapsto Q(q,p)=q+\bar Q(\tau_q\o,p).$
Note that the map $\bar p$ is defined on the set $\bar A$ so that $v\in [\bar Q(\o,-1),\bar Q(\o,1)]$.  
The following explicit description is needed in upcoming proofs.

\begin{prop} \label{easy:lem4}
Write $Q^{\pm}(\o)=\bar Q(\o,\pm 1)$ and set
\begin{eqnarray} \label{key2:for1} 
\mathcal{L}( \omega,v):=\int_{ Q^-(\omega)}^v
\bar P(\o,\bar p(\o,a))\ {\rm d}a- Q^-(\omega).
\end{eqnarray}
Then $\cL(\o,v)$ is a generating function of $F$ of complexity  $0$.
\end{prop}

\begin{proof}
We prove it if $F$ is positive monotone; the negative monotone case is similar. 
From (\ref{key2:for1}) we deduce that the corresponding $\cG(q,Q;\o)=\cL(\tau_q\o,Q-q)$ is
equal to $$\int_{q+ Q^-( \tau_q\omega)}^Q
P(q,\, p(q,\tilde{Q}))\, {\rm d} \tilde{Q} - Q^{-}( \tau_q\omega)$$
which is equal to
\begin{eqnarray} \label{key2:for} \nonumber
&&\int_{q+ Q^-( \tau_q\omega)}^Q
\left(P(q,\, p(q,\tilde{Q}))+1\right)\, {\rm d} \tilde{Q} -(Q-q)\nonumber \\
&=&
\int_q^{Q+q^{-}(\tau_Q\o)}\left(p(\tilde q,Q)+1\right){\rm d}\tilde q-(Q-q)\nonumber \\
&=&
\int_q^{Q+q^{-}(\tau_Q\o)}p(\tilde q,Q){\rm d}\tilde q+q^{-}(\tau_Q\o).
\end{eqnarray}
For the first equality in (\ref{key2:for}), we used  that $F$ is area\--preserving.
Here $F^{-1}(Q,P)=(q(Q,P),p(Q,P))$ and $q^{\pm}$ is defined by
$q(Q,\pm 1)=Q+q^{\pm}(\tau_Q\o)$ so that $Q\mapsto Q+q^{\pm}(\tau_Q\o)$
is the inverse of the map $q\mapsto q+Q^{\pm}(\tau_q\o)$.
Applying the Fundamental Theorem of Calculus to (\ref{key2:for}) we obtain 
that 
$
\mathcal{G}_Q(q,\,Q)=P(q,\,p)$ and $- \mathcal{G}_q(q,\,Q)=p$. 
Then (\ref{key:for}) follows.
\end{proof}

\begin{figure}[htbp]
\psfrag{F}{$F$}
\psfrag{A}{{\small $\mathcal{G}(q,\,Q)$}}
\psfrag{L1}{$\ell_{-}$}
\psfrag{L0}{$\ell_{+}$}
\psfrag{F-1}{$F^{-1}$}
\psfrag{q}{$q$}
\psfrag{Q}{$Q$}
\psfrag{q1}{$q(Q,1)$}
\psfrag{q0}{$q(Q,-1)$}
\psfrag{Q1}{$Q(q,1)$}
\psfrag{Q0}{$Q(q,-1)$}
  \begin{center}
    \includegraphics[width=11cm]{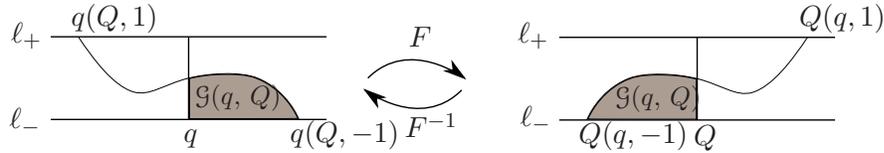}
   \caption{{\small Area\--preserving random twist
$F \colon \mathcal{S}  \to  \mathcal{S}$ 
and inverse. The
   area of the shaded regions is  
$\mathcal{G}(q,\,Q)$ in (\ref{key2:for}). }}
    \label{Lfig1}
  \end{center}
\end{figure}

\subsection{\textcolor{black}{Fixed points}}

 The following implies the $N=0$ statement in Theorem~\ref{epbt0}.

 \begin{prop} \label{prop:key}
 Let $ F \colon \mathcal{S} \times \Omega \to \mathcal{S}$ be an area\--preserving random montone twist with generating function  $\mathcal{L} \colon \bar{A} \to \mathbb{R}$. 
 Then $\psi \colon \bar{A}_0 \to \mathbb{R}$ given by
$\psi(a,\, \omega)=\bar\psi(\tau_a\o):=\mathcal{L}( \tau_a\omega,\,0)$
has infinitely many critical points. Furthermore,  
except for degenerate cases, $\psi$ has maximum and minimum critical points. 
In degenerate cases $\psi$ has a continuum of critical points. If  $\psi$ is bounded and non\--constant, it 
oscillates infinitely many times, so it has maximums and minimums.
\end{prop}
 
 \begin{proof}
We prove the last statement by contradiction. Suppose that $\psi(a,\, \omega)$ is monotone for large $a$. Then
 $
 \lim_{a \to \infty}\psi(a,\, \omega)=\psi(\infty,\, \omega)
 $ 
 is well\--defined. By
 ergodicity $\psi(\infty,\, \omega)=\psi(\infty)$ is independent of $\omega$. On 
the other hand, for any bounded continuous function $J \colon \mathbb{R} \to \mathbb{R}$ we have that
$\mathbb{E}\, J(\psi(a,\, \omega))=\mathbb{E}\, J(\bar\psi( \omega))$
for every $a$, and therefore $J(\psi(\infty))=\mathbb{E}\, J(\bar\psi( \omega)).$ 
Thus $\bar\psi( \omega)=\psi(\infty)$ a.s. In other words, if $\psi(a,\, \omega)$
doesn't oscillate, then $\psi(a,\, \omega)$ is constant.
 \end{proof}

\subsection{\textcolor{black}{Construction of random monotone twists and spectral nature of
fixed points}} \label{mono}

As we argued in Proposition~\ref{easy:lem4}, a monotone twist map may be determined in terms of its generating function. 
We now explain how we can start from a scalar-valued function $H(\o,v)$ and construct a monotone twist map from it.
To explain this construction, let us derive a useful property of generating functions.  Recall $Q^{\pm}(\o)=\bar Q(\o,\pm 1)$.

\begin{prop}\label{prop9.1} Let $\cL(\o,v)$ be as in Proposition~\ref{easy:lem4}. Then the function
\begin{equation}\label{eq9.1}
\cL(\o,Q^+(\o))-Q^+(\o),
\end{equation}
is constant and $\cL(\o,Q^-(\o))=Q^-(\o)$.
\end{prop}

\begin{proof} From
$
 F(q,-\mathcal{G}_q(q,Q;\o))=(Q,\mathcal{G}_Q(Q,q;\o)),
$
we deduce
$$
\bar F(\o,\cL_v(\o,v)-\cL_\o(\o,v))=(v,\cL_v(\o,v)).
$$
Since $P=\pm 1$ if and only if $p=\pm 1$, we obtain
$
\cL_\o(\o,Q^{\pm}(\o))=0$ and $\cL_v(\o,Q^{\pm}(\o))=\pm 1$.
But
$$
\nabla_\o\left(\cL(\o,Q^\pm(\o))\right)=\cL_\o(\o,Q^{\pm}(\o))+\cL_v(\o,Q^{\pm}(\o))Q^{\pm}_\o(\o)
=\pm Q^{\pm}_\o(\o),$$ 
which means that the function $\cL(\o,Q^\pm(\o))\mp Q^\pm(\o)$ is constant by the ergodicity of $\bP$.
On the other hand, by the definition of $\cL$  (see (\ref{key2:for1})) we know that $\cL(\o,Q^-(\o))=-Q^-(\o)$. 
\end{proof}

We are ready to give a recipe for constructing a monotone twist map from a ${\rm C}^2$ function
$
H:\O\times \bR\to\bR,
$
which satisfy the following conditions
\begin{align}\label{eq9.2}
\left\{ \begin{array}{rl}
&H(\o,0)=0,\ \ \ H(\o,a)>0 {\text{ for }}a>0,\\\\
&\eta(\o)=\inf\{a>0\,\, | \,\, H(\o,a)=2\}<+ \i,
          \end{array} \right .
\end{align}
almost surely. For such a function $H$, we set  $\s(\o)=\eta(\o)-\frac 12\int_0^{\eta(\o)} H(\o,a){\rm d}a$ and
\begin{eqnarray}
Q^-(\o)&=&-\s(\o),\ \ \ \ Q^{+}(\o)=(\eta-\s)(\o);\label{eq9.3}\\
\bar G(\o,v)&=&H(\o,v+\s(\o)),\ \ \ G(q,Q;\o)=\bar G(\tau_q\o,Q-q); \label{9.3'} \\
\cL(\o,v)&=&\int_0^{v+\s(\o)}H(\o,a){\rm d}a-v;\,\,\,\,\,\,
 \cG(q,Q;\o)=\cL(\tau_q\o,Q-q). \nonumber
\end{eqnarray}

\begin{theorem}\label{prop9.2} Assume that $H \colon \Omega \times \mathbb{R} \to \mathbb{R}$ 
satisfies \eqref{eq9.2} and the condition $G_q<0$ with $G$ defined as in
\eqref{9.3'}. Then there exists a unique monotone twist map $F$ such that 
$
F(q,-\mathcal{G}_q(q,Q))=(Q,\mathcal{G}_Q(q,Q)),
$
and $F(q,\pm 1)=(q+Q^{\pm}(\tau_q\o),\pm 1)$ with $Q^{\pm}$ defined by \eqref{eq9.3}.
Moreover, if $\bar q$ is a local maximum (respectively minimum) for $q\mapsto \psi(q)=\cG(q,q)$, then
${\rm D}F$ at the $F$-fixed point $(\bar q,-\mathcal{G}_q(\bar q,\bar q))$ has negative (respectively positive) eigenvalues.
\end{theorem}

\begin{proof} 
By the definition, 
$$
\cG(q,Q)=\int_{q+Q^-(\tau_q\o)}^QG(q,Q')\ {\rm d}Q'-(Q-q),
$$
which implies
\begin{eqnarray} \label{that}
\mathcal{G}_Q=G-1,\ \ \ \cG_{Qq}=G_q<0.
\end{eqnarray}
From \eqref{that} we learn that the map $Q\mapsto \mathcal{G}_q(q,Q)$ is decreasing and, as a result,
the equation 
\begin{eqnarray}\label{cal}
\mathcal{G}_q(q,Q)=-p
\end{eqnarray}
may be solved for $Q$, to yield a $p$\--increasing function $Q=Q(q,p)$. We set 
$
P(q,p)=\mathcal{G}_Q(q,Q(q,p))=G(q,Q(q,p))-1,
$ 
so that
$F(q,p)=(Q(q,p),P(q,p))$. Note that the monotonicity condition is satisfied 
because $Q$ is increasing in $p$. 
We need to show that the boundary conditions are satisfied and that $F$ is area-preserving. For the latter,
observe that by differentiating both sides of the relationship \eqref{cal}, we obtain
$\cG_{qq}+\cG_{Qq}Q_q=0$, $\cG_{qQ}Q_p=-1$, 
$P_q=\cG_{Qq}+\cG_{QQ}Q_q$, and $P_p=\cG_{QQ}Q_p$.
It follows that
\begin{equation}\label{eq9.4}
 {\rm D}F=-\cG_{Qq}^{-1}
\begin{bmatrix}
{\cG_{qq}} & {1} \\
{\cG_{qq}\cG_{QQ}-\cG_{qQ}^2}& \cG_{QQ}
\end{bmatrix} .
\end{equation}
It follows from \eqref{eq9.4}  that if the eigenvalues of ${\rm D}F$ are $\l$ and $\l^{-1}$, then $\l>0$ if and only if
$
{\rm Trace}({\rm D}F)=\frac {\cG_{qq}+\cG_{QQ}}{-\cG_{qQ}}=\l+\l^{-1}\geq 2.
$
Equivalently ${\rm D}F$ has positive eigenvalues if and only if
$
\psi''(q)=(\cG_{qq}+\cG_{QQ}+2\cG_{qQ})(q,q)>0.
$
The case of negative eigenvalues may be treated in the same way.

For the boundary conditions, we first establish
\begin{equation}\label{eq9.5}
\cL_\o(\o,Q^{\pm}(\o))=0,\ \ \ \cL_v(\o,Q^{\pm}(\o))=\pm 1. 
\end{equation}
For the second equality in \eqref{eq9.5}, observe that $\cL_v=\bar G-1$, and by definition
$\bar G(\o,Q^-(\o))=H(\o,0)=0$, and 
$ \bar G(\o,Q^+(\o))=H(\o,Q^+(\o)-Q^-(\o))=H(\o,\eta(\o))=2$. As for the first
 equality in \eqref{eq9.5}, observe that by the definition of $\s$, $G$ and $\cL$,
\begin{align*}
\cL(\o,Q^{-}(\o))+Q^{-}(\o)&=0,\\
\cL(\o,Q^{+}(\o))-Q^{+}(\o)&=\int_0^{Q^+(\o)+\s(\o)}H(\o,a){\rm d}a-2Q^+(\o)\\
&=
\int_0^{\eta(\o)}H(\o,a){\rm d}a-2(\eta-\s)(\o)=0.
\end{align*}
As a result 
\begin{eqnarray} \label{x:eq}
\cL(\o,Q^{\pm}(\o))\mp Q^{\pm}(\o)=0.
\end{eqnarray}
 Differentiating (\ref{x:eq}) with respect to $\o$ yields
$
0=\cL_\o(\o,Q^{\pm}(\o))+\cL_v(\o,Q^{\pm}(\o))Q^{\pm}_\o(\o)\mp Q^{\pm}_\o(\o) =\cL_\o(\o,Q^{\pm}(\o))$,
which is precisely the first equality in \eqref{eq9.5}. 

We are now ready to verify the boundary conditions.
We wish to show that $Q(q,\pm 1)=q+Q^{\pm}(\tau_q\o)$, or equivalently
\[
\pm 1=-\mathcal{G}_q(q,q+Q^{\pm}(\tau_q\o))=(\cL_v-\cL_\o)(\tau_q\o,Q^{\pm}(\tau_q\o)).
\]
This is an immediate consequence of \eqref{eq9.5}. It remains to verify $P(q,\pm 1)=\pm 1$. 
We certainly have
\begin{eqnarray*}
P(q,\pm 1)=\mathcal{G}_Q(q,q+Q^{\pm}(\tau_q\o)) =G(q,q+Q^{\pm}(\tau_q\o))-1
=\bar G(\tau_q\o,Q^{\pm}(\tau_q\o))-1 \nonumber 
\end{eqnarray*}
This and \eqref{eq9.5} imply $P(q,\pm 1)=\pm 1$, because $\bar G-1=\cL_v$.
\end{proof}
\begin{remark}
\normalfont
 $\s$ in \eqref{eq9.3} is motivated by \eqref{eq9.1}. It is chosen so that 
 $\cL(\o,Q^+(\o))=Q^{+}(\o).$
\end{remark}
\begin{remark}
\normalfont
The monotonicity condition $G_q=\cG_{Qq}<0$ may be expressed  as
$
H_\o(\o,a)<H_a(\o,a)(1-\s'(\o)).
$
The derivative of $\s$ may be calculated with the aid of \eqref{eq9.3}:
\begin{eqnarray}
\s'(\o)=\eta'(\o)-\frac 12H(\o,\eta(\o))\eta'(\o)-\frac 12\int_0^{\eta(\o)} H_\o(\o,a){\rm d}a 
=-\frac 12\int_0^{\eta(\o)} H_\o(\o,a) {\rm d}a. \nonumber
\end{eqnarray}
\end{remark}

\subsection{\textcolor{black}{The density of fixed points}} \label{sd}

When $F$ is a positive twist map,  it has a generating function $\cG(q,Q,\o)=\cL(\tau_q\o,Q-q)$ and any fixed point of $F$ is of the form $(q_0, \cL_v(\tau_{q_0}\o,0))$ where $q_0$ is a critical point of the random process
$\psi(q,\o)=\bar\psi(\tau_q\o)$ (Propositions \ref{easy:lem3} and \ref{easy:lem4}). 
We have also learned that any random process $\psi$ has
 infinitely many local maximums and minimums. In this section we give sufficient conditions to ensure that such a random process has a positive density of critical points, which in turn yields a positive density for fixed points of a monotone twist map. Let $\sharp B$ be the cardinality of a set $B$.
 
 \begin{definition} \label{density}
 The  \emph{density} of   $A\subset \bR$ is 
  ${\rm den}(A):=\lim_{\ell\to\i} \,\, (2\ell)^{-1}\sharp(A\cap[-\ell,\ell]).$
\end{definition}

Let us state a set of assumptions for the random process $\psi(q,\o)=\bar\psi(\tau_q\o)$ 
that would guarantee
the existence of a density for the set $Z(\o):=\{q \,\,| \,\,\psi'(q,\o)=0\}.$

\begin{hyp} \label{hypo}
\begin{itemize}
\item[(i)]  $\psi(q,\o)$ is twice differentiable almost surely and if
\[
\phi_\ell(\d;\o)=\sup\Big\{|\psi''(q,\o)-\psi''(\hat q,\o)| \,\,\, | \,\,\, q,\hat q\in[-\ell,\ell],\ |q-\hat q|\leq \d\Big\},
\]
then
$
\lim_{\d\to 0}\bE \ \phi_\ell(\d;\o)=0
$
for every $\ell>0$.
\item[(ii)] The random pair $(\bar \psi_\o(\o),\bar\psi_{\o\o}(\o))$ has a probability density $\rho(x,y)$. In other words,
for any bounded continuous function $J(x,y)$,
\[
\bE J(\psi'(q,\o),\psi''(q,\o))=\int_{\mathbb{R}} J(x,y)\,\rho(x,y)\,{\rm d}x {\rm d}y.
\]
\item[(iii)] 
There exists $\e>0$ such that
$\rho(x,y)$ is jointly continuous for $x$ satisfying $|x|\geq \e$.
\end{itemize}
\end{hyp}

We define $
\bar Z(\o):=\{q \, \, | \,\, \psi'(q,\o)=0,\ \psi''(q,\o)\neq 0\}$ and
$N_\ell(\o):=\bar Z(\o)\cap [-\ell,\ell]$. 
It is well known that if we assume Hypothesis \ref{hypo}, then
\begin{equation}\label{eq9.6}
\bE \ N_\ell(\o)={2\ell}\int_{\mathbb{R}} \rho(0,y)|y|\ {\rm d}y.
\end{equation} 
This is the celebrated \emph{Rice Formula} and its proof can be found in \cite{AhSt2000,AW}.  
Next we state a direct consequence of Rice Formula and the 
Ergodic Theorem.

\begin{theorem}\label{prop9.5} If  $\psi$ satisfies Hypothesis~\ref{hypo} then $\bar Z(\o)=Z(\o)$
almost surely and
\begin{equation}\label{eq9.7}
\lim_{\ell\to\i}\bE\left| \frac 1{2\ell} N_\ell(\o)-\int_{\mathbb{R}} \rho(0,y)|y|\ {\rm d}y\right|=0.
\end{equation}
\end{theorem}

\begin{proof}
Pick a smooth function $\zeta:\bR\to[0,\i)$ such that its support is contained in the interval
$[-1,1]$, $\zeta(-a)=\zeta(a)$,  and  $\int_{\mathbb{R}}\zeta(q) {\rm d}q=1$. Set 
$\zeta_\e(q):=\e^{-1}\z(q/\e).$
It is not hard to show
\begin{equation}\label{eq9.8}
\frac 1{2\ell}N_\ell(\o)\geq  \frac 1{2\ell}\int_{-\ell+\e}^{\ell-\e}\left|\zeta'_\e*\hat\psi(q,\o)\right|{\rm d}q=:X_\e(\ell,\o),
\end{equation}
where $\hat \psi(q,\o)=1\!\!1(\psi'(q,\o)>0)$ (this is \cite[Lemma~3.2]{AW}). We note that if 
$
\eta_\e(\o)=\left|\int_{\mathbb{R}} \zeta'_\e(a)\hat\psi(a,\o)\ {\rm d}a\right|,
$
then 
\begin{eqnarray}
\eta_\e(\tau_q\o)=\left|\int_{\mathbb{R}} \zeta'_\e(a)\hat\psi(a,\tau_q\o)\ {\rm d}a\right|
=\left|\int_{\mathbb{R}} \zeta'_\e(a)\hat\psi(a+q,\o)\ {\rm d}a\right|=\left|\zeta'_\e*\hat\psi(q)\right|. \label{this}
\end{eqnarray}
From \eqref{this} and the Ergodic Theorem we deduce
\begin{equation}\label{eq9.9}
\lim_{\ell\to\i} \frac 1{2\ell}\int_{-\ell}^{\ell}\left|\zeta'_\e*\hat\psi(q,\o)\right|{\rm d}q=\bE \eta_\e,
\end{equation}
almost surely and in the  ${\rm L}^1(\bP)$ sense. 

On the other hand,
\begin{equation}\label{eq9.10}
\lim_{\e\to 0}\bE \eta_\e=\int_{\mathbb{R}} \rho(0,y)|y|\ {\rm d}y=:\bar X.
\end{equation}
This follows the proof of Rice Formula, see \cite[proof of Theorem 3.4]{AW}. 

Again by Rice Formula,
$
0=\bE   \left[\frac 1{2\ell} N_\ell(\o)-\bar X\right]
=\bE\left[\frac 1{2\ell} N_\ell(\o)-X_\e(\ell,\o)\right]-
\bE\left[X_\e(\ell,\o)-\bar X\right], \nonumber
$
which implies
\begin{equation}\label{eq9.11}
\lim_{\e\to 0}\, \limsup_{\ell\to\i}\bE\left[\frac 1{2\ell} N_\ell(\o)-X_\e(\ell,\o)\right]=0,
\end{equation}
because by \eqref{eq9.9} and \eqref{eq9.10},
\begin{equation}\label{eq9.12}
\lim_{\e\to 0}\, \limsup_{\ell\to\i}\bE\left|X_\e(\ell,\o)-\bar X\right|=0.
\end{equation}
From \eqref{eq9.8} and \eqref{eq9.11} we deduce
\begin{eqnarray} \label{u}
 \lim_{\e\to 0}\limsup_{\ell\to\i}\bE\left|\frac 1{2\ell} N_\ell(\o)-X_\e(\ell,\o)\right|=0.
\end{eqnarray}
Then \eqref{u} and \eqref{eq9.12} imply \eqref{eq9.7}.
\end{proof}

\section{\textcolor{black}{Complexity $N=1$ area\--preserving random twists}} \label{ergsec}

This section proves a result which implies the $N=1$ case in Theorem~\ref{epbt0}, \ref{+2}. A result concerning the
spectral nature of the fixed points is also proven.

\subsection{\textcolor{black}{Domain of random generating functions} } 

We begin by describing the domain the random generating function of a complexity one twist.

\begin{lemma} \label{above}
 Let $F$ be an area\--preserving random twist of complexity one with decomposition
$F=F_1\circ F_0,$
 where $F_1$ is a positive monotone area\--preserving random twist and $F_0$ is negative monotone area\--preserving random twist.  Let $\mathcal{G}^0,\, \mathcal{G}^1$ be the generating
 functions, respectively, of the monotone twists $F_0,\, F_1$. 
Then $G_1:=F_1^{-1}$ is a negative
 area\--preserving random twist with
 generating function given by 
$\hat{\mathcal{G}}^1(q,\, \xi):=-\mathcal{G}^1(\xi,\, q),$
and  if 
 \begin{eqnarray} \label{D_1}
 D_0:=\op{Domain}({\mathcal{G}^0})\,\,\, 
\textup{and}\,\,\, D_1:=\op{Domain}(\hat{\mathcal{G}}^1),
 \end{eqnarray}
 then we have a proper inclusion of sets $D_0 \subsetneq D_1$ (see Figure~\ref{Lfig4}). 
\end{lemma}

\begin{proof}
Note that $G_1(a,\,\pm 1)=(Q^\pm_1(a),\,\pm 1)$ and $F_0(a,\,\pm 1)=(Q^\pm_0(a),\,\pm 1),$
 with $\pm (Q^\pm_i(a)-a)<0$ and $Q^\pm_i$ increasing. Since $F$ is an area\--preserving random twist map, 
we may write $F^{-1}(q,\,\pm1)=(\hat{Q}^\pm(q),\,\pm1)$ with $\hat{Q}^\pm$ increasing and such that
 $
\pm(\hat{Q}^\pm(q)-q)<0
 $
 for all $q$. For $i=0,\,1$ let 
 $
 \partial^\pm D_i=\{(a, Q^\mp_i(a)) \,\, | \,\, a \in \mathbb{R}\}
 $
 denote the boundary curves of $D_i$.  From $G_1=F_0\circ F^{-1}$, we deduce
  $Q^\pm_0 (\hat{Q}^\pm(q))=Q^\pm_1(q),$
  and therefore 
  \begin{eqnarray} \label{n1}
  Q^-_0(q)<Q^-_1(q)
  \end{eqnarray}
and 
\begin{eqnarray} \label{n2}
Q^+_0(q)>Q^+_1(q).
\end{eqnarray}
Then \eqref{n1} (respectively \eqref{n2}) implies that  the upper (respectively lower) boundary of $D_1$ is strictly
 above (respectively below) $D_0$. It follows that $D_0 \subsetneq D_1$, as desired. 
 \end{proof}

\begin{figure}[htbp]
\psfrag{D1}{$D_0$}
\psfrag{D2}{$D_1$}
\psfrag{GR}{$\nabla {\cI}$}
\psfrag{pD1+}{$\partial^+ D_0$}
\psfrag{pD1-}{$\partial^- D_0$}
\psfrag{beta}{$\,$}
\psfrag{grQ_1}{$\textup{Graph}(Q^0)$}
\psfrag{Q2}{$\textup{Graph}(Q^1)$}
  \begin{center}
    \includegraphics[width=6cm]{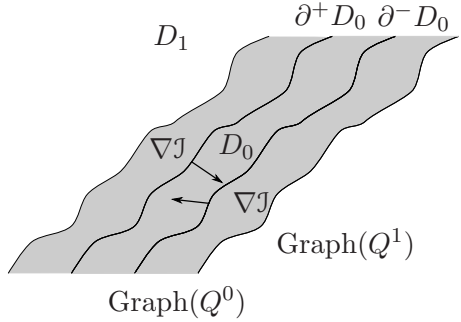}
   \caption{{\small The domains $D_0$ and  $D_1$ and the gradient
    $\nabla {\cI}$.}}
    \label{Lfig4}
  \end{center}
\end{figure}

 \subsection{\textcolor{black}{Gradients and geometry of domains}}
  
  Let $D_0$ be defined by \eqref{D_1}.
  
 \begin{cor} \label{2above}
 The map
 \begin{eqnarray} \label{m}
{\cI}(q,\, \xi):=\mathcal{G}^0(q,\, \xi)+\mathcal{G}^1(\xi,\,q)
 \end{eqnarray}
 is well\--defined on the set $D_0$, cf. \eqref{D_1}.
\end{cor}

\begin{proof}
If $(\xi,\, q) \in D_0 \cap D_1$ then the sum 
$\mathcal{G}^0(q,\, \xi)+\mathcal{G}^1(\xi,\, q)$ is well defined. 
The corollary follows from Lemma \ref{above}.
 \end{proof}

 \begin{lemma} \label{gradlem}
 The gradient  $\nabla {\cI}$ of 
${\cI} \colon D_0 \to \mathbb{R}$ is inward on  $\partial^{\pm}D_0$ and $\mp \cI_\xi,\pm\cI_q >0$ on $\partial^{\pm}D_0$.
 \end{lemma}
 
 \begin{proof}
 If $F_0(q,\,p)=(\xi,\, \eta)$ and $F_1(\xi,\, \eta')=(q,\, P)$, then
 ${\cI}_q(q,\, \xi)=P-p$ and ${\cI}_\xi(q,\, \xi)=\eta-\eta'$
 hold.
We express the domain $D_0$  of ${\cI}$ given by (\ref{D_1}) as 
$
\{(\xi, \,q) \,\,\, | \,\,\, p=p(q, \, \xi)=- \mathcal{G}^0_q(q,\, \xi) \in [-1,\,1]\}. 
 $
  On $\partial^{-}D_0$,
 $\eta=p=1$ and $P,\, \eta'<1$ (because $D_0 \subsetneq D_1$).  
So on $\partial^{-}D_0$ we have $ {\cI}_\xi(q,\, \xi)>0$ and 
 $ {\cI}_q(q,\, \xi)<0$. 
 On $\partial^+D_0$ we have $\eta=p=-1$ and $\eta',\, P<1$. So on $\partial^{+}D_0$ we have ${\cI}_\xi(q,\, \xi)<0$ and  
 $ {\cI}_q(q,\, \xi)>0$.  
 The lower boundary $\partial^{-}D_0$ is the graph of an increasing function $q \mapsto h(q)$, and of course $h'(q)>0$. So, the
 tangent to $\partial^{-}D_0$ is $(1,\, h'(q))$ and the inward normal
 is $(-h'(q),\,1)$.  On $\partial^{-}D_0$ we have $ {\cI}_\xi(q,\, \xi)>0$ and 
 $ {\cI}_q(q,\, \xi)<0$. So we have that the dot product
 $
 \langle ( {\cI}_q(q,\, \xi),\, {\cI}_\xi(q,\, \xi)),\,\, 
 (-h'(q),\,1) \rangle=
 -h'(q) {\cI}_q(q,\, \xi)+  {\cI}_\xi(q,\, \xi)>0.
 $
 That is, on the lower boundary $\nabla {\cI}$ is inward. 
 
The case of the upper boundary is analogous.
 \end{proof}

 \subsection{\textcolor{black}{Fixed points}}

 If we set $\hat{D}:=\{(q,\,a)\, | \, (q,\, q+a) \in D_0\},$ we have that, for a pair
 of random processes $B^-(\tau_q\omega), B^+(\tau_q\omega)>0$,
 $
 \hat{D}=\{(q,\,a) \,\, | \,\, -B^-(\tau_q\omega)<a<B^+(\tau_q\omega)\}.
 $
We then use the notation of Lemma~\ref{genf:def} to set
 $
 \bar{{\cI}}( \tau_q\omega,a):=\cL^0(\tau_q\o,a)+\cL^1(\tau_a\tau_q\o,-a)={\cI}(q,\, q+a).
 $
Define the map $\bar K \colon \Omega \times [-1,\,1] \to \mathbb{R}$ by
$
 K(q,p;\o)=\bar K(\tau_q\omega,p)=\bar{\cI}\left(\tau_q\omega, B(\tau_q\omega,p)\right) ,
$
where
$
 B(\tau_q\omega,p)=\frac{p+1}2 B^+(\tau_q\omega)+\frac {p-1}2\,B^-(\tau_q\omega).
$
Note that 
\begin{align} \nonumber
 K_p(q,p;\o)=&\frac 12 \bar{\cI}_a\left(\tau_q\omega, B(\tau_q\omega,p)\right)
\left(B^+(\tau_q\omega)+B^-(\tau_q\omega)\right), \\ \label{eq8.2}
K_q(q,p;\o)=&  \bar{\cI}_\o\left(\tau_q\omega, B(\tau_q\omega,p)\right)
  +\bar{\cI}_a\left(\tau_q\omega, B(\tau_q\omega,p)\right)B_\o(\tau_q\o,p). 
 \end{align}
Hence there is a one-one correspondence between the critical points of $K$ and $\cI$. 
From \eqref{eq8.2} and Lemma~\ref{gradlem} we conclude the following.

\begin{lemma} \label{gradlem2}
 The gradient  $\nabla {K}$ of $K \colon \mathcal{S} \times \O \to \mathbb{R}$
is inward on  the boundary of $\mathcal{S}$.
 \end{lemma}

 The following result implies the case $N=1$ in Theorem~\ref{epbt0}. 
  
 \begin{theorem}  \label{keyresult}
 Let $\bar K \colon \O  \times [-1,\,1] \to \mathbb{R}$ be a $\op{C}^1$\--map
 such that
$\mp\bar K_p(\cdot,\pm 1)>0$.
 Let $K(q,p;\o):=\bar K(\tau_q\omega,p).$ 
 \begin{enumerate}[{\rm (a)}]
 \item \label{xa}
 $K$ has infinitely many critical points; 
 \item \label{xb}
 Furthermore,  the critical points of $K$ occur as follows:
 \begin{itemize}
 \item[{\rm (1)}]
 Either $K$ has a continuum of critical points;
 \item[{\rm (2)}]
Or  $K$ has both infinitely many local maximums, and infinitely many saddle points or  local minimums.
 \end{itemize}
 \end{enumerate}
 \end{theorem}
 
 \begin{proof}
 We prove (\ref{xb}).  If $\hat{K}(\omega):=\max_{a\in[-1,1]}\bar K(\omega,\, a),$ 
 then either $\hat{K}$
 is constant or $\hat{K}(\tau_q\omega)$ oscillates almost surely. In the former case
 for almost all $\omega$, there exists $a(\omega)$ such that 
$\bar K(\omega,\,a(\omega) )$
 is a maximum and (of course) $a(\omega) \notin \{-1,\, 1\}$ by the assumption 
$\mp\bar K_p(\cdot,\pm 1)>0$. 
More concretely, we set
$
a(\o)=\max\{p\in[-1,1] \,\, | \,\,  \bar K(\o,p)=\hat K(\o)\}.
$
Hence $K$ has a continuum of critical points of the form $\{(q,a(\tau_q\o))\,\, | \,\,q\in\bR\}.$
In the latter case, there are infinitely many local maximums.
 Choose $\bar q$ so that $\hat{K}(\tau_{\bar q}\omega)$ is a local maximum. 
For such $(\bar q,\, \omega)$
 choose $a(\tau_{\bar q}\omega)$ so that 
$\bar K(\tau_{\bar q}\omega, \,  a(\tau_{\bar q}\omega))=\hat K(\tau_{\bar q}\o)$.  
Therefore $ K$  has infinitely many local
maximums by Proposition \ref{prop:key}. 

 Note that if
$$
\Omega_0:=\Big\{\omega \,\,\, | \,\,\, \{\tau_a\omega \, | \, a>a_0\} \,
\textup{is dense for every }a_0\, \Big\}, 
$$
then $\mathbb{P}(\Omega_0)=1$. This is true because the family $\{\tau_a:a\in\bR\}$ is ergodic
and by assumption $\bP(U)>0$ for every open set
$U$. Given $\o\in \O_0$,  consider the ordinary differential equation with initial value condition
 \begin{eqnarray}
    \label{equ:(3)}
    \left\{
      \begin{aligned}
        q'(t)&\,=\bar K_{\omega}( \tau_{q(t)}\omega, \, p(t))\\
         p'(t)&\,=\, \bar K_p( \tau_{q(t)}\omega, \, p(t))\\
          q(0)&\,=\,0,\ \ \ \ p(0)=a.     \end{aligned}\right.
  \end{eqnarray}
There are two possibilities; the first
 possibility is that for some $a$, we have that $q(t)$ is unbounded as
 $t \to \infty$, and in this case we claim that
        there is a continuum of critical points. The second possibility is that
         $q(t)$ is \emph{always} bounded as $t \to \infty$, 
and in this case we claim that $K$ has either  infinitely many
saddle points or local minimums. We  proceed with case by case.

\paragraph{\emph{Case 1}} (\emph{The map $q(t)$ is unbounded as $t \to \infty$ for some $\o\in\O_0$}). 
We want to prove that
$K$ has a continuum of critical points.  Define 
$\omega(t):=\tau_{q(t)}\omega,$
and let  $\phi^r$ be the flow of (\ref{equ:(3)}). Note that 
$
\frac{\op{d}}{\op{d}\!t} \bar  K(\omega(t),\,p(t) )= | \nabla \bar K (\omega(t),\,p(t) ) |^2 \geq 0.
$
Since $q(t)$ is unbounded,  $\omega(t)$ can approach almost any point in $\Omega$.   Moreover if $\tau_{q(t_n)}\omega \to \overline{\omega}$ and
$p(t_n) \to \overline{p}$, then we claim that
$\nabla\bar K(\overline{\omega},\overline{p})=0$. Indeed, 
if $\lambda:=\sup_{t>t_0} \bar K(\omega(t),\,p(t) ),$ we have
$\lambda=\bar K(\overline{\omega},\, \overline{p})$, and since
$$
\lambda=\sup_{t>t_0}\bar K(\omega(t+r),\, p(t+r)),
$$
 we have, for any $r>0$, that
$\lambda=\bar K(\overline{\omega},\, \overline{p})=
\bar K(\phi^r(\overline{\omega},\, \overline{p})).$
Hence $\nabla \bar K(\overline{\omega}, \, \overline{p})=0$; otherwise
$
\frac{\op{d}}{\op{d}\!r} \bar K(\phi^r(\overline{\omega},\,
\overline{p}))|_{r=0}>0,
$
which is impossible. Note that $\overline{\omega}$ could be
any point in $\Omega$ and therefore
for such $\overline{\omega}$ there exists
$\overline{p}=\overline{p}(\overline{\omega})$ such that
$\nabla K(\overline{\omega},\, \overline{p}(\overline{\omega}))=0,$ i.e. we
have
a continuum of critical points. This concludes Case 1.

\paragraph{\emph{Case 2}} (\emph{The map $q(t)=q(t,\, \omega)$ is bounded for every
 $\omega\in \O_0$}). 
We claim that if $\bar K$ does not have a continuum of fixed points, then $K$
has infinitely many critical points which  are local minimums or 
saddle points. Suppose that this is not the case, then we want to arrive at a contradiction. 
In order to do this let $\bar x=(\overline{q},\, \overline{p})$ be a local maximum, which we know it always
exists by the paragraphs preceding Case 1. In fact we may take a $\d>0$ such that $K(x)\leq K(\bar x)$ for every 
$x=(q,p)$ with $q\in(\bar q-\d,\bar q+\d)$.
Now take a closed curve $\gamma$ such 
that $(\overline{q},\, \overline{p})$ is inside $\gamma$ and if $a \in \gamma$, then 
 $
\lim_{t \to  \infty} \phi^t(a) = (\overline{q},\, \overline{p})=\overline{a}.
 $
 For example, we may take $\gamma$ to be part of level set of the function
 $(q,\, p) \mapsto K(q,\,p)$ with value $c<K(\bar x)$ very close
 to $K(\bar x)$.  Since $K$ does not have a continuum of critical points, we may choose such level set $\g$ such that
 $K$ has no critical point on $\g$. From this latter property 
 we deduce that $\g$ is homeomorphic to a circle.  
\begin{figure}[htbp]
\psfrag{L1}{$\ell_1$}
\psfrag{L0}{$\ell_{-1}$}
\psfrag{a}{$a$}
\psfrag{a'}{$a'$}
\psfrag{a''}{$a''$}
\psfrag{ga}{$\Gamma(a)$}
\psfrag{ga'}{$\Gamma(a')$}
\psfrag{ga''}{$\Gamma(a'')$}
\psfrag{gamma}{$\gamma$}
\psfrag{hata}{$\overline{a}$}
  \begin{center}
    \includegraphics[width=4.5cm]{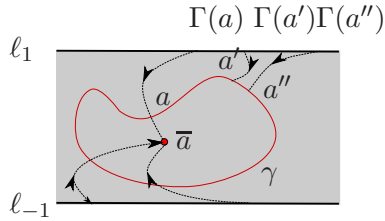}
   \caption{{\small Note that $a \in \gamma$ while $\overline{a}$ is enclosed by $\gamma$.}}
    \label{Lfig6}
  \end{center}
\end{figure}
 Let $a \in \gamma$. If there is no other type of critical points, then the curve
 $t \mapsto \phi^t(a)$, where $t \leq 0$, must reach the boundary for some $t_a<0$, because
 $
 \frac{\op{d}}{\op{d}\!t} K(\phi^t(a)) \geq 0.
 $
This defines a map
$
\Gamma \colon \gamma \to  ( \mathbb{R}\times \{-1\}) \cup (\mathbb{R} \times \{1\}) ,\,\,\,\, 
\Gamma(a):=  \phi_{t_a}(a).
$ 
We now argue that in fact $\G$ is continuous. To show the continuity of $\G$ at $a\in\g$, extend $K$ continuously near $\G(a)$, choose $\e>0$ and 
set $$\eta=(\phi^\th(a)\,\,\,|\,\,\,\th\in[t_a-\e,\e]).$$
Choose $\e$ sufficiently small so that $\phi^\th(a)$ is inside $\g$ for $\th\in (0,\e]$,
and $\phi^t(a)$ is outside the strip for $t\in({t_a-\e},t_a)$.
Choose $\hat a\in\g$ close to $a$ so that 
$
\eta'=(\phi^\th(b)\,\,\, | \,\,\,\th\in[t_a-\e,\e])
$
is uniformly close to $\eta$.
Since $\phi^{t_a}(\hat a)$ is near $\G(a)$,
we can choose $\hat a$ close enough to $a$ to guarantee that $\G(\hat a)$ is close to $\G(a)$. Moreover, we can easily show that $\G(c)$ is between $\G(a)$ and $\G(\hat a)$ for any $c$ between 
$a$ and $\hat a$ on $\g$. Hence $\G$ is a homeomorphism from a neighborhood of $a$ onto its image. 
Since $\gamma$ is homeomorphic to $S^1$, its
homeomorphic image $\Gamma(\gamma)$ cannot be fully contained inside of $  \mathbb{R}\times\{-1\}\cup\mathbb{R}\times \{+1\}$. Therefore there exists 
 $a\in\g$ such that any limit point $z$ of $\phi^t(a)$ as $t\to-\i$  is a critical point inside the strip that is not a local maximum. Clearly $z\notin (\bar q-\d,\bar q+\d)$. Let us assume for example that $z=(q_1,p_1)$ with $q_1>\bar q+\d$. Take another local maximum $\hat x=(\hat q,\hat p)$ to the right of $\bar x$ and assume that 
 $
 K(\hat x)\geq K(x)$ for all $x\in (\hat q-\hat\d,\hat q+\hat \d)\times[-1,1]$. 
 Since $\phi^t(a)$ cannot enter  
$(\hat q-\hat\d,\hat q+\hat \d)\times[-1,1]$ we deduce that $q_1\in(\bar q+\d,\hat q-\hat d)$. 

Repeating the above argument for other local maximums, we deduce that there exist infinitely critical points in between local maximums that are not local maximums.
\end{proof}

\subsection{\textcolor{black}{Nature of the fixed points in terms of generating function}} \label{nature}

A result similar to Theorem~\ref{prop9.2} holds for complexity $N=1$ twist maps.

\begin{theorem}\label{prop9.7} Let $F$ and $\cI$ be as in Lemma~\ref{above} and Corollary~\ref{2above}.

Let $(\bar q,\bar\xi)$ be a critical point of $\cI$ and $\vec x$ be the corresponding 
fixed point of $F$ as in Proposition~\ref{easy:lem3}. Assume that $\cI_{\xi\xi}(\bar q,\bar \xi)\neq 0$. Then ${\rm D}F(\vec x)$ has positive (respectively negative) eigenvalues
if and only if $\det \cI(\bar q,\bar \xi)\geq 0$ (respectively $\leq 0$).
\end{theorem}

\begin{proof}
Recall that $\cS(q,Q;\xi)=\cS^0(q,\xi)+\cS^1(\xi,Q)$ and:
$$
\cG_{\xi}(q,Q;\xi)=0 \Rightarrow F(q,\, -\mathcal{G}_q(q,Q;\xi))=(Q,\mathcal{G}_Q(q,Q;\xi)).
$$
Observe that if $cI(\bar q,\bar \xi)=
cG_{\xi\xi}(\bar q,\bar q;\bar \xi)\neq 0$, then near $(\bar q,\bar q,\bar \xi)$,  we can solve 
$\cG_{\xi}(q,Q;\xi)=0$ as $\xi=\xi(q,Q)$. Write
$\cT(q,Q)=\cG(q,Q;\xi(q,Q))$. Then
$
\cT_q=\mathcal{G}_q,\,\,\,\, \cT_Q=\mathcal{G}_Q,
$
and
$
F(q,-\cT_q(q,Q))=(Q, \cT_Q(q,Q)).
$
As a result, we can show
\begin{eqnarray}
{\rm D}F=\frac{1}{-\cT_{qQ}}
\begin{bmatrix}
\cT_{qq} & {1} \\
\cT_{qq}T_{QQ}-\cT_{qQ}^2& \cT_{QQ} \nonumber
\end{bmatrix},
\end{eqnarray}
in the same way we derived \eqref{eq9.4}.
Observe that 
$
{\rm Trace}({\rm D}F)=\frac{\cT_{qq}+\cT_{QQ}}{-\cT_{qQ}}.
$
Since $\cT_{qq}=\cG_{qq}+\cG_{q \xi}\xi_q$,
$\cT_{QQ}=\cG_{QQ}+\cG_{Q\xi }\xi_Q$, 
$\cT_{qQ}=\cG_{qQ}+\cG_{q\xi}\xi_Q$, and
$\cT_{Qq}=\cG_{Qq}+\cG_{Q\xi}\xi_q$, we have that
$$
\cT_{qq}+\cT_{QQ}+2\cT_{qQ}=\cG_{qq}+\cG_{QQ}+2\cG_{qQ}+(\cG_{q\xi}+\cG_{Q\xi })(\xi_q+\xi_Q).
$$
On the other hand, by differentiating the relationship $\cG_\xi(q,Q;\xi(q,Q))=0$, we have
$
\cG_{\xi q}+\cG_{\xi\xi} \xi_q=0$ and $\cG_{\xi Q}+\cG_{\xi\xi}\xi_Q=0$,
or equivalently,
$
\xi_q=-\frac{\cG_{\xi q}}{\cG_{\xi\xi}}$, $\xi_{Q}=-\frac{\cG_{\xi Q}}{\cG_{\xi\xi}}$. In particular,
$
\cG_{\xi q}+\cG_{\xi_Q}+\cG_{\xi\xi}(\xi_q+\xi_Q)=0,
$
which in turn implies
$$
\cT_{qq}+\cT_{QQ}+2\cT_{qQ}=\cG_{qq}+\cG_{QQ}+2\cG_{qQ}-\frac{1}{\cG_{\xi\xi}}(\cG_{q \xi}+\cG_{Q \xi})^2. 
$$
Furthermore, if
$\mathcal{I}(q,\xi)=\cG(q,q;\xi)$, then
$\mathcal{I}_q=\cG_q+\mathcal{G}_Q$, $\mathcal{I}_{\xi}=\cG_{\xi}$,  and
\begin{eqnarray}
{\rm D}^2\mathcal{I}=\begin{bmatrix}
\cG_{qq} +\cG_{QQ}+2\cG_{qQ} & \cG_{\xi Q}+\cG_{\xi q} \\
\cG_{\xi Q}+\cG_{\xi q} & \cG_{\xi \xi} \nonumber
\end{bmatrix}.
\end{eqnarray}
So
$
\cT_{qq}+\cT_{QQ}+2\cT_{qQ}=\frac{{\rm det}({\rm D}^2\mathcal{I})}{\cG_{\xi\xi}}.
$
Also,
$
\cT_{qQ}=\cG_{qQ}-\frac{\cG_{q\xi}\cG_{Q\xi}}{\cG_{\xi \xi}}.
$
Now
\begin{eqnarray} \label{also}
{\rm Trace}({\rm D}F)-2=\frac{\cT_{qq}+\cT_{QQ}+2\cT_{qQ}}{-\cT_{qQ}}=\frac{{\rm det}({\rm D}^2\mathcal{I})}{\cG_{qQ}\cG_{\xi\xi}-
\cG_{q\xi}\cG_{Q\xi}}.
\end{eqnarray}
Recall $\cG(q,Q;\xi)=\cG^0(q,\xi)+\cG^1(\xi,Q)$ with $\cG^0_{q\xi}>0$, and
$\cG^1_{Q\xi}<0$
because $F^0$ is a negative monotone twist and $F^1$ is a positive monotone twist. Hence we obtain 
$-\cG_{q\xi}\cG_{Q\xi}>0$. 
On the other hand $\cG_{qQ}=0$, which simplifies (\ref{also}) to
$$
{\rm Trace}({\rm D}F)-2=\frac{\cT_{qq}+\cT_{QQ}+2\cT_{qQ}}{-\cT_{qQ}}=\frac{{\rm det} ({\rm D}^2\mathcal{I})}{-\cG_{q\xi}\cG_{Q\xi}}. 
$$
This expression has the same sign as ${\rm det}({\rm D}^2\mathcal{I})$. Finally
${\rm D}F$ has positive eigenvalues if and only if ${\rm Trace}({\rm D}F)\geq 2$,  if and only if 
${\rm det}({\rm D}^2 \mathcal{I})\ge 0$, which concludes the proof.
 \end{proof}

\section{\textcolor{black}{Complexity $N=2$ area\--preserving random twists}}

In this section we settle the case $N=2$ in Theorem~\ref{epbt0}.

\subsection{\textcolor{black}{Domain of random generating functions}}

Next we describe the domain of a random generating function associated to
a complexity $N=2$ twist.

\begin{lemma} \label{lem9.2}
 Let $F$ be an area\--preserving random twist of complexity $N=2$.  Suppose that $F$ decomposes as $ F= F_2\circ F_1\circ F_0,
 $
 where $F_1$ is a positive monotone area\--preserving random twist and $F_j$ is negative monotone area\--preserving random twist for $j=0,2$.
 Let $\mathcal{G}^0,\cG^1, \mathcal{G}^N$ be the corresponding generating
 functions. Write $G_i=F^{-1}_i$ and define $Q_i^{\pm}$ and $\hat Q_i^{\pm}$ by
$F_i(q,\pm 1)=(Q_i^{\pm}(q),\pm 1)$ and $G_i(q,\pm 1)=(\hat Q_i^{\pm}(q),\pm 1)$. Then
the function
 $\cI(q,\xi_1,\xi_2):=\cG^0(q,\xi_1)+\cG^1(\xi_1,\xi_2)+\cG^2(\xi_2,q),$
is well-defined on the set
$$D=\left\{(q,\xi_1,\xi_2) \,\,\,|\,\,\, Q_0^+(q)\leq \xi_1\leq Q_0^-(q),\ \   \hat Q_2^-(q)\leq \xi_2\leq \hat Q_2^+(q)\right\},$$
Moreover, if $(q,\xi_1,\xi_2)\in D$, then $Q^-_1(\xi_1)<\xi_2< Q^+_1(\xi_1).$
\end{lemma}

\begin{proof} Since $F_1=G_2\circ F\circ G_0$, we have
\begin{equation}\label{eq9.30}
\hat Q_2^{\pm}\circ Q^{\pm}\circ \hat Q^{\pm}_0=Q_1^{\pm},
\end{equation}
where $Q^\pm$ are defined by the relationship  $F(q,\pm 1)=(Q^\pm(q),\pm 1)$. On the set $D$,  
$\cG^0(q,\xi_1)$ and $\cG^2(\xi_2,q)$ are well defined. It is sufficient to check that if $(q,\xi_1,\xi_2)\in D$, then 
$\cG^1(\xi_1,\xi_2)$ is well-defined. That is, 
$
Q^-_1(\xi_1)<\xi_2< Q^+_1(\xi_1).
$
To see this observe that by \eqref{eq9.30},
\begin{eqnarray*}
\pm Q_1^{\pm}(\xi_1)=\pm\left(\hat Q_2^{\pm}\circ Q^{\pm}\circ \hat Q^{\pm}_0\right)(\xi_1) 
\geq \pm\left(\hat Q_2^{\pm}\circ Q^{\pm} \right)(q)> \pm\hat Q_2^{\pm}(q)\geq \pm\xi_2,
\end{eqnarray*} 
as desired. Here for the first inequality we used the fact that $Q^\pm$ and $\hat Q^\pm_2$ 
are increasing and that in $D$, we have 
 $\hat Q_0^-(\xi_1)\leq q \leq \hat Q_0^+(\xi_1);$
for the second inequality we used 
 $\pm Q^{\pm}(q)> \pm q,$
which concludes the proof.
\end{proof}

We define $B_0^\pm(\o),B_2^\pm(\o)> 0,$ by 
$
Q_0^\pm(q)=q\mp B^\pm_0(\tau_q\o)$ and 
$\hat Q_2^\pm(q)=q\pm B_2^\pm(\tau_q\o)$. Let
\begin{eqnarray} \label{mapKX}
K(q,p;\o)=\bar K(\tau_q\o,p)=\cI(q,\xi(q,p))=\bar\cI(\tau_q\o,q+\bar\xi(\tau_q \o,p)),
\end{eqnarray}
where $p=(p_1,p_2)$, $\bar\xi( \o,p)=(\bar\xi_1(\o,p_1),\bar\xi_2(\o,p_2)),$ 
$\xi(q,p)=(q+\bar\xi_1(\tau_q\o,p_1),q+\bar \xi_2(\tau_q\o,p_2)),
$ 
and $\bar\xi_1$ and $\bar\xi_2$ are defined  by
$
\bar\xi_1(\o,p_1):=\frac {p_1+1}2B^-_0(\o)+  \frac {p_1-1}2B^+_0(\o)$ and
$\bar\xi_2(\o,p_2):=\frac {p_2+1}2B^+_{2}(\o)+
\frac {p_2-1}2B^-_{2}(\o)$.

\begin{lemma} \label{lem9.2} Let $K:\bR\times [-1,1]^2\times\O\to\bR$ be as in \eqref{mapKX}.
The following hold:
\begin{enumerate}[{\rm (i)}]
\item \label{x}
 There exists a one-to-one correspondence between critical points of $\cI$ and $K$. 
\item \label{xx}
The vector $\nabla K$ is pointing inward on the boundary of $\bR\times [-1,1]^2$. 
\end{enumerate}
\end{lemma}

\begin{proof}
Evidently $K(q,p_1,p_2)=K(q,p;\o)$ satisfies
\begin{eqnarray}\label{eq:33}
\left\{ \begin{array}{rl}
 K_{p_1}(q,p_1,p_2) = \frac 12\cI_{\xi_1}(q,\xi(q,p))\left(B_0^++   B_0^-\right)(\tau_q\o),\\
 K_{p_2}(q,p_1,p_2) = \frac 12 \cI_{\xi_2}(q,\xi(q,p))\left( B_2^++   B_2^-\right)(\tau_q\o),\\
K_{q}(q,p_1,p_2)= \cI_{q}(q,\xi(q,p))+\cI_{\xi_1}(q,\xi(q,p))+\cI_{\xi_2}(q,\xi(q,p))\\
+\cI_{\xi_1}(q,\xi(q,p))\left(\frac {p_1+1}2\nabla B^-_0+  \frac {p_1-1}2
\nabla B^+_0\right)(\tau_q\o)\\
+\cI_{\xi_2}(q,\xi(q,p))\left(\frac {p_1+1}2\nabla B^-_2+  \frac {p_1-1}2
\nabla B^+_2\right)(\tau_q\o).
          \end{array} \right . 
\end{eqnarray}
It follows from \eqref{eq:33} that there exists a one-to-one correspondence between the critical points of $\cI$ and $K$
because $B^\pm_i>0$ for $i=0,2$.  This proves \eqref{x}.

We now examine the behavior of $K$ across the boundary.
Observe that the functions $K_{p_1}$ and $\cI_{\xi_1}$
(respectively  $K_{p_2}$ and $\cI_{\xi_2}$) have the same sign. Moreover,
\begin{eqnarray}
p_1&=&\pm 1\Leftrightarrow \xi_1=Q_0^{\mp}(q),\nonumber \\
p_2&=&\pm 1\Leftrightarrow q=Q_2^{\pm}(\xi_2). \nonumber
\end{eqnarray}
It remains to verify  
\begin{align}
\xi_1&=Q_0^{\mp}(q)\Rightarrow  \pm\cI_{\xi_1}< 0, \nonumber \\
q&=Q_2^{\pm}(\xi_2)\Rightarrow \pm\cI_{\xi_2}< 0 . \nonumber
\end{align}
Let us write $\xi_0$ for $q$ and $\xi_3$ for $Q$. We define functions $p^i(\xi_i,\xi_{i+1})$ and 
$P^i(\xi_i,\xi_{i+1})$ by
$
F^i\left(\xi_i,p^i(\xi_i,\xi_{i+1})\right)=\left(\xi_{i+1}, P^i(\xi_i,\xi_{i+1})\right).
$
We then have
$\cI_{\xi_1}=\cG^0_{Q}+\cG^1_q=P^0-p^1$ 
and $\cI_{\xi_2}=\cG^1_{Q}+\cG^2_q=P^1-p^2$. Finally we assert,
\begin{align*}
p_1&=\pm 1\Rightarrow \xi_1=Q^{\mp}_0(q)\Rightarrow p^0=P^0=\mp1\Rightarrow \pm\cI_{\xi_1}<0,\\
p_2&=\pm 1\Rightarrow \xi_2=\hat Q^{\pm}_2(q)\Rightarrow p^2=P^2=\pm1\Rightarrow \pm\cI_{\xi_2}<0,
\end{align*}
as desired. Here we are using the fact that if $p^0=P^0=\mp1$ or $p^2=P^2=\pm1$, then 
$Q^-_1(\xi_1)<\xi_2< Q^+_1(\xi_1)$ or equivalently $p^1,P^1\notin\{-1,1\}.$
\end{proof}

\subsection{\textcolor{black}{Fixed points}}

 The following result implies the complexity $N=2$ statement in Theorem~\ref{epbt0}. The proof of  
 is sketched because it is similar to that of Theorem~\ref{keyresult}.
  
 \begin{theorem}  \label{keyresult2}
 Let $K:\bR\times [-1,1]^2\times\O\to\bR$, and
$
K(q,\,p;\o):=\bar K(\tau_q\omega,\,p)
$
be ${\rm C}^1$ up to the boundary with $\nabla K$ pointing inwards on the boundary. Then
\begin{enumerate}[{\rm (a)}]
\item\label{xxa}
 $K$ has infinitely many critical points. 
 \item \label{xxb}
 The critical points of $K$ occur as follows:
 \begin{itemize}
 \item[{\rm (1)}]
 Either $K$ has a continuum of critical points;
 \item[{\rm (2)}]
 Or $K$ has both infinitely many local maximums, and infinitely many saddle points or  local minimums.
 \end{itemize}
 \end{enumerate}
 \end{theorem}

\begin{proof} We prove (\ref{xxb}). As in the proof of Theorem~\ref{keyresult}, we assume that $K$ does not have a continuum of critical points and deduce that $K$ has infinitely many isolated local maximums. The 
$q$ component of the flow remains bounded
almost surely. We take a local maximum $a$ and a connected component $\g$ of a level set of $K$ 
associated with a regular value $c$ of $K$, very close to the value $K(a)$. The surface $\g$ is an oriented
 closed manifold and if $K$ has no other type of critical point, then
$\G:\g\to \bR\times\partial [-1,1]^2,$ is a homeomorphism from $\g$ onto its image. Since the set
$\bR \times \partial [-1,1]^2$ cannot contain a homeomorphic image of $\g$, we arrive at a contradiction. From this 
we  deduce the conclusion of the theorem as in the proof of Theorem~\ref{keyresult}.
\end{proof}

\section{\textcolor{black}{Complexity $N\geq 3$ area\--preserving random twists}} \label{ergsec2}
 
 We prove the $N\geq 3$ case of Theorem~\ref{epbt0}, item (\ref{+2}). The results in
 Sections~\ref{t1} and \ref{t2} hold for general $N\geq 0$. The other results
 use  that $N$ is at least $3$.
 
\subsection{\textcolor{black}{Geometry of the domain of the generating function}}   \label{t1} 
 
 Let $F$ be an area\--preserving random twist of complexity $N$. 
As in Theorem~\ref{dec}, we assume that $N$ is an odd number and that $F$ decomposes as
in \eqref{eq4.1}. Recall that $\mathcal{G}^0,\dots, \mathcal{G}^N$ denote the generating
 functions, respectively, of the monotone twists $F_0,\dots, F_N.$ 
 Set 
$$\cI(q,\xi)=\cG(q,q;\xi)=\cL(\tau_q\o,0;\xi-q),\ \ \ \cI'(q,\eta)=\cI(q,\eta+q)=:\bar\cI(\tau_q\o,\eta),$$ 
where $\cG$ and $\cL$ are defined by Lemma~\ref{genf:def}, and $\eta+q=(\eta_1+q,\dots,\eta_N+q)$.  
Given a realization $\o$, we write $D=D(\o)$ for the domain of the definition of $\cI$. We also set
$D'(\o)=\{\eta\in\bR^N\ |\ (0,\eta)\in D(\o)\}$ so that the domain of the function $\cI'$ is exactly
$\{(q,\eta)\ |\ \eta\in D'(\tau_q\o)\}$.
To simplify the notation, we write $\xi_0$ for $q$ and $\xi_{N+1}$ for $Q$.
In this way, we can write
$
F^i(\xi_i, p^i)=(\xi_{i+1},P^i),
$
where 
$p^i=p^i(\xi_i,\xi_{i+1})=-\cG_q^i(\xi_i,\xi_{i+1})$ and
$P^i=P^i(\xi_i,\xi_{i+1})=\cG_Q^i(\xi_i,\xi_{i+1})$. 
Here by $\cG_q^i$ and $\cG_Q^i$ we mean the partial 
derivatives of $\cG^i$ with respect to its first and second arguments respectively.
As before, we write $G^i$ for the inverse of $F^i$ and define increasing functions $Q_i^\pm$ and $\hat Q^\pm_i$ by
$
F^i(q,\pm 1)=(Q^\pm_i(q),\pm 1)$ and
$G^i(q,\pm 1)=(\hat Q^\pm_i(q),\pm 1)$. Let
\begin{align}
E(\xi_1,\xi_N)=& \bigcap_{i=1}^{N-1}\Big\{(\xi_2,\dots,\xi_{N-1})\, |\, (-1)^{i+1}Q^-_{i}(\xi_i)\leq (-1)^{i+1}\xi_{i+1}\leq 
 (-1)^{i+1}Q^+_{i}(\xi_i) \Big\}. \nonumber
\end{align}
Then the set $D$ consists of points $(q,\xi)$ such that
$\xi_1\in[Q^+_0(q),Q^-_0(q)]$, $\xi_{N}\in[\hat Q^+_N(q),\hat Q^-_N(q)]$ and 
$(\xi_2,\dots,\xi_{N-1})\in E(\xi_1,\xi_N)$. Alternatively, we can write
\begin{align*}
E(\xi_1,\xi_N)=\bigcap_{i=1}^{N-1}\Big\{(\xi_2,\dots,\xi_{N-1})\, |\, -1\leq p^i(\xi_i,\xi_{i+1}),\,\,\,  P^i(\xi_i,\xi_{i+1})\leq 1 \Big\}.
\end{align*}
We write $\partial D=\partial^+ D\cup\partial^- D,$ where
$\partial^+ D$ and $\partial^- D$ represent the upper and lower boundaries of $D$. Then
$
\partial^+ D=\bigcup_{i=0}^N\partial^+_i D$ and  $\partial^- D=\bigcup_{i=0}^N\partial^-_i D$, 
where
\begin{align*}
\partial^\pm_0 D&=\left\{(q,\xi)\in D\, |\, \xi_1=Q^\mp_0(q)\right\} 
=\left\{(q,\xi)\in D\, |\, p^0(q,\xi_1)=P^0(q,\xi_1)=\mp 1\right\},\\
\partial^\pm_N D&=\left\{(q,\xi)\in D\, |\, \xi_N=\hat Q^\mp_N(q)\right\}
=\left\{(q,\xi)\in D\, |\, p^N(\xi_N,q)=P^N(\xi_1,q)=\mp 1\right\}, \\
\partial^\pm_i D&=\left\{(q,\xi)\in D\, |\, \xi_{i+1}=Q^\pm_i(\xi_{i}) \right\} \
{\text{ for $i$ odd and  }}1< i<N,\\
\partial^\pm_i D&=\left\{(q,\xi)\in D\, |\, \xi_{i}=\hat Q^\pm_i(\xi_{i+1}) \right\} \
{\text{ for $i$ even and  }}1< i<N.
\end{align*}
We also write $\partial_i D=\partial^+_i D\cup \partial^-_i D$,
$\bar\partial^\pm_i D=\partial^\pm_i D\setminus\left(\partial^\pm_{i-1} D\cup\partial^\pm_{i} D\right)$,
$\bar\partial^\pm_0 D=\partial^\pm_0 D\setminus\left(\partial^\pm_{1} D\cup\partial^\pm_{N} D\right)$,
$\bar\partial^\pm_N D=\partial^\pm_N D\setminus\left(\partial^\pm_{0} D\cup\partial^\pm_{N-1} D\right)$.

\subsection{The gradient} \label{t2}

Next examine the behavior of $\nabla\cI$ across the boundary. The randomness of $D(\o)$
and $\cI$ play no role and the proof is analogous in the periodic case (\cite{Go2001}).

\begin{prop} \label{above'}
 Let $ F= F_N\circ \dots\circ F_1\circ F_0$ be an area\--preserving random 
 twist decomposition as in \eqref{eq4.1}.
Then the following properties hold.
\begin{enumerate}[{\rm (P.i)}]
\item \label{item:i}
If $ 1< i <N$ is even, $\nabla \cI$ 
 is inward  along $\bar\partial^\pm_iD$;
\item \label{item:ii}
If $ 1< i <N$ is odd, $\nabla \cI$ 
 is outward  along $\bar\partial^\pm_iD$;
\item \label{item:iii}
$\nabla \cI$ 
 is outward  along $\bar\partial^\pm_N D$;
\item \label{item:iv}
 $\nabla \cI$ 
 is inward  along $\bar\partial^\pm_0D$.
\end{enumerate}
\end{prop}

\begin{proof} 
Evidently,
 ${\cI}_q(q,\, \xi)=P^N-p^0$ and ${\cI}_{\xi_i}(q,\, \xi)=P^{i-1}-p^i,$
 for $i=1,\dots,N$.
We wish to study the behavior of the function $\cI$ across the boundary of $D$. 
On $\partial^{\pm}_0D$, we have $p_0=P_0=\mp 1$. Since $\cI_{\xi_1}=P^{0}-p^1$, we deduce
\begin{equation}\label{eq10.1}
\pm\cI_{q}>0,\ \ \ \pm\cI_{\xi_1}<0 \ \ \ {\text{ on }}\ \ \bar\partial^\pm_0D.
\end{equation} 
On $\partial^{\pm}_ND$, we have $p_N=P_N=\mp 1$. Since $\cI_{\xi_N}=P^{N-1}-p^N$, we deduce 
\begin{equation}\label{eq10.2}
\pm\cI_{\xi_{N-1}}<0,\ \ \  \pm\cI_{\xi_N}>0 \ \ \ {\text{ on }}\ \ \bar\partial^\pm_ND.
\end{equation}
On $\partial^{\pm}_iD$ we have $P^i=p^i=\pm (-1)^{i+1}$; hence
$
 \pm (-1)^{i}{\cI}_{\xi_i}(q,\, \xi)=\pm (-1)^{i}(P^{i-1}-p^i)\geq 0$ and
$\pm (-1)^{i}{\cI}_{\xi_{i+1}}(q,\, \xi)=\pm (-1)^{i}(P^{i}-p^{i+1})\leq 0$ if $1<i<N$. 
The inequalities are strict on $\bar\partial^{\pm}_iD$.   

 The boundary $\partial^{\pm}_0D$ is the set of points $(q,\xi)$ such that $\xi_1=Q^\mp_0(q)$ with
$q\mapsto Q^\mp_0(q)$ increasing.
 So, if we write $\dot Q^\mp_0(q)$ for the derivative of $Q^\mp_0(q)$,
then any vector that has $(1,\,\dot Q^\mp_0(q))$ for its projection onto $(q,\xi_1)$-space would be tangent to $\partial^{\pm}_0D$. Hence a vector $n_0$ that has
 $\pm(\dot Q^\mp_0(q),\,-1)$ for the first two components and $0$ for the other components, is
an inward normal vector to the $\partial^{\pm}_0D$ part of boundary.
   As a result, we have that on $\partial^{\pm}_0D$
 $
  \langle  \nabla\mathcal{I},\, n_0  \rangle=
  \pm \left(\dot Q^\mp_0(q){\cI}_{q}-{\cI}_{\xi_{1}}\right)  > 0,
 $
by \eqref{eq10.1}. Here $\langle\cdot,\cdot\rangle$ denotes the dot product.
That is, on $\bar \partial^{\pm}_0D$, the gradient $\nabla {\cI}$ is  inward, proving (P.\ref{item:iv}).
Similarly we use \eqref{eq10.2} to establish  (P.\ref{item:iii}).

Assume that $i$ is odd. The boundary $\partial^{\pm}_iD$ is the set of points $(q,\xi)$ such that the components
$\xi_i$ and $\xi_{i+1}$ lie on the graph $\xi_{i+1}=Q_i^{\pm}(\xi_i)$.
Again, if we write $\dot Q_i^{\pm}$ for the derivative of $Q_i^{\pm}$,
then any vector that has  $(1,\, \dot Q_i^{\pm}(\xi_i))$ for its projection onto $(\xi_i,\xi_{i+1})$-space
 would be tangent to $\partial^{\pm}_iD$. As a result, the vector $n_i$ that has
$\pm( \dot Q_i^{\pm}(\xi_i),\,-1)$ for $(i,i+1)$ components and $0$ for the other components, is
an inward normal to the $\partial^{\pm}_iD$
portion of the boundary.
 Hence on $\dot Q_i^{\pm}$,
 $
  \langle  \nabla\mathcal{I},\, n_i  \rangle=
  \pm \left(\dot Q_i^\pm{\cI}_{\xi_i}-{\cI}_{\xi_{i+1}}\right)  < 0,
 $
proving (P.\ref{item:ii}).  (P.\ref{item:i}) is established similarly.
  \end{proof}

Define 
$
\partial_{\rm in}D:=\left\{x\in \partial D\, |\, \nabla\cI(x) {\text{ is inward}}\right\}$, and similarly define
$$\partial_{\rm out}D:=\left\{x\in \partial D\, |\, \nabla\cI(x) {\text{ is outward}}\right\}.$$ We write 
$\bD^k$ for the $k$-dimensional unit ball.

With the same proof as Gole \cite{Go2001}, Proposition \ref{above'} implies the following lemma.

\begin{lemma}\label{lem10.2}
Suppose that $N=2k+1$ with $k\ge 1$. Then the sets $\partial_{\rm out}D$ and
$\partial_{\rm in}D$ are homeomorphic to $\bR\times \bD^{k+1}\times \partial \bD^{k}$
and  $\bR\times \partial\bD^{k+1}\times  \bD^{k}$ respectively.
\end{lemma}

\subsection{\textcolor{black}{Fixed points}}
 Now we  prove a result which implies the $N\geq3$ case in Theorem~\ref{epbt0}.

\begin{theorem} \label{keyresult3}
Let $Z(\o)$ be the set of critical points
of $\cI$ and $\hat Z:=\{q\, |\, (q,\xi)\in Z(\o)\}$. 
Then:
\begin{enumerate}[{\rm (a)}]
\item \label{xxxa}
 $\sup\hat Z=+\i$
and
$\inf \hat Z=-\i$ with probability $1$;
\item \label{xxxb}
$\cI$ has infinitely many critical points in $D$ almost surely.
\end{enumerate}
\end{theorem}

\begin{proof}
(\ref{xxxb}) follows from (\ref{xxxa}). Consider the ordinary differential equation
 \begin{eqnarray}
    \label{equ:(6)}
    \left\{
      \begin{aligned}
        q'(t)&\,=\,\cI_q(q(t), \, {\xi}(t);\o)\,=\,\bar {\cI}_\o(\tau_{q(t)}\o, \, {\xi}(t)),\\
         {\xi}'(t)&\,=\, {\cI}_{\xi}(q(t),\, {\xi}(t);\o)\,=\,\bar {\cI}_{\xi}(\tau_{q(t)}\o, \, {\xi}(t)).           
       \end{aligned} \right. \nonumber
  \end{eqnarray}
Now we distinguish two cases  (in analogy with the proof of Theorem \ref{keyresult}).

\paragraph{Case 1} (\emph{The map $q(t)$ is unbounded either as $t \to \infty$ or $t \to -\infty$}). 
Analogously to Case 1 in Theorem \ref{keyresult}, we are assuming that for a realization $\o\in\O_0$, either $(x(t)=(q(t),\xi(t)):t\ge 0)$ or $(x(t)=(q(t),\xi(t)):t\le 0)$ remains inside the domain 
$D(\o)$ and the $q$-component is unbounded. As in the proof of Case 1 in Theorem \ref{keyresult},
we can show that for all
$\omega \in \Omega$ there exists ${\xi}(\omega)$ such that $(\omega, \, {\xi}(\omega))$ is 
a critical point for $\bar {\cI}$. In particular ${\cI}$ has a continuum of critical points.

\paragraph{Case 2} (\emph{The map $q(t)$ is always bounded as $t \to \pm \infty$}). We want
to show that ${\cI}$ has critical points strictly inside of $D=D(\o)$. Let us first assume by contradiction that  ${\cI}$ has
no critical point inside of $D(\o)$ for a realization of $\o$. 
Consider the flow 
$
\phi^t(q,\,{\xi}):=({q(t)},\, {\xi}(t))=x(t),
$
which starts at the point $x=(q,{\xi} )\in \partial_{in}D$.
 Since $q(t)$ stays bounded and we are assuming that there is no critical
point inside, the flow must exit at some positive time $e(x)$. Write $\hat \phi(x)=\phi^{e(x)}(x)$.
 Note that the sets 
$\partial_{\rm in}D$ and $\partial_{\rm out}D$ are open relative to $\partial D$. 
We now argue that the function $\hat \phi(x)$ is continuous. For example, $\hat \phi$ is continuous at $x$ Simply because
 we may extend ${\cI}$ near $\hat\phi(x)$ across the boundary so that for some small $\e>0$,
 the flow $\phi^t(x)$ is well-defined and lies outside $D$ for $t\in (e(x),e(x)+\e)$. We can then guarantee that $\phi^t(y)$ is close to $\phi^t(x)$
for $t\in [0,e(x)+\e)$ and $y$ sufficiently close to $x$. As a result, for $y$ sufficiently close to $x$, the point
$\phi^{e(y)}(y)$ is close to $\phi^{e(x)}(x)$, concluding the continuity of $\hat\phi$. 
In fact by interchanging $\partial_{\rm out}D$ with $\partial_{\rm in}D$, we can show that $\hat \phi^{-1}$ is continuous. As a result $\hat\phi$ is a homeomorphism from $\partial_{\rm in}D$ onto $\partial_{\rm out}D$. This is impossible because $\partial_{\rm in}D$ is not homeomorphic to $\partial_{\rm out}D$by Lemma~\ref{lem10.2}.  Hence 
$\cI$ has at least one critical point in ${\rm Int}(D)$ and $Z(\o)\neq \emptyset$.

It remains to show that the set $Z(\o)$ is unbounded on both sides. We only verify the unboundedness from above as the 
boundedness from below can be established in the same way. 
Suppose to the contrary that $Z(\o)$ is bounded above with positive probability. Since 
\begin{equation}\label{eq9.78}
Z(\tau_q\o)=\tau_{-q}Z(\o)=\{(a-q,\xi)\,|\, (a,\xi)\in Z(\o)\},
\end{equation}
 by stationarity, we learn that the set $Z(\o)$ is bounded above almost surely.
Define $\bar x(\o)=(\bar q(\o),\bar {\xi}(\o))$ by
$
\bar q(\o)=\max\{q\, |\, (q,{\xi})\in Z(\o)\}$ and $\bar {\xi}(\o)=\max\{{\xi}\, |\, (\bar q(\o),{\xi})\in Z(\o)\}$. 
Again by \eqref{eq9.78},
$
\bar q(\tau_a\o)+a=\bar q(\o)$ and $\bar {\xi}(\tau_a\o)=\bar {\xi}(\o)$, 
for every $a\in \bR$. As a result,
$
\bP\big(\bar q(\o)\geq \ell\big)=\bP\big(\bar q(\tau_a\o)+a\geq \ell\big) 
=\bP\big(\bar q(\tau_a\o)\geq \ell-a\big)=\bP\big(\bar q(\o)\geq \ell-a\big),
$
for every $a$ and $\ell$. Since this is impossible unless $\bar q=\i$, we deduce that the set $Z(\o)$ is unbounded from above.
\end{proof}

\section{\textcolor{black}{Appendix A: Poincar\'e\--Birkhoff Theorem (1912)}} \label{sec:classical}

A diffeomorphism
$
F \colon \mathcal{S} \to \mathcal{S}$,
$$F(q,p)=(Q(q,p),P(q,p)),$$
 is an \emph{area\--preserving periodic twist} if:
 \begin{itemize}
\item[(1)] {\emph{area preservation}}: it preserves area; 
\item[(2)]
{\emph{boundary invariance}}: it preserves $\ell_{\pm}:=\mathbb{R}\times \{\pm 1\} $, i.e. 
 $P(q,\pm 1):=\pm 1$; 
 \item[(3)]
{\emph{periodicity}}: $F(q+1,p)=(1,0)+F(q,p)$ for all $p,q$;
\item[(4)] {\emph{boundary twisting}}:  
$F$ is orientation preserving and  
$\pm Q(q,\pm 1)>\pm q$ for all $q$. 
\end{itemize}
To emphasize the analogy with Section~\ref{mainsec},  
we may alternatively replace (3) by 

(3'): $F(q,\,p)=(q+\bar Q(q,\,p),\bar P(q,\,p))$ for a map
$
\bar F:=(\bar Q,\bar P) \colon \mathcal{S} \to \mathcal{S}
$
such that $\bar F(q+1,\,p)=\bar F(q,\,p)$ for all $(q,\,p)$, and (4) by 

(4'): $q\mapsto Q(q,\pm 1)$ is increasing  and  
$\pm Q(q,\pm 1)>\pm q$ for all $q$.  
 \begin{figure}[H]
  \begin{center}
    \includegraphics[width=4.5cm]{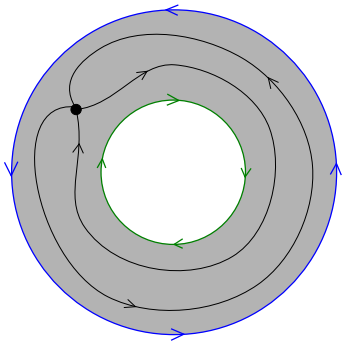}
   \caption{{\small Fixed point of an area\--preserving twist defined by a  flow.}}
    \label{annulus}
  \end{center}
\end{figure}

\begin{theorem}[Poincar\'e\--Birkhoff] \label{cpbt}
An area\--preserving periodic twist 
$
F \colon \mathcal{S} \to \mathcal{S}
$
has at least two geometrically distinct fixed points.
\end{theorem}

Theorem~\ref{cpbt} was proved\footnote{One can use symplectic dynamics
to study area\--preserving maps, see~\cite{BrHo2012}.  Section 3.4
of Bramham et al. \cite{BrHo2012} proves Theorem~\ref{cpbt} using the important tool of 
finite energy foliations \cite{HWZ03}.}
 in certain 
cases by Poincar\'e \cite{Po12}. Later Birkhoff gave a full proof  
and presented  generalizations \cite{Bi1,Bi2}; in \cite{Bi3} he explored its applications to dynamics.  
See~\cite[Section 7.4]{Ba1997} and  \cite{BN} for 
an expository account. 

Arnol'd formulated
the higher dimensional analogue of Theorem~\ref{cpbt}: the
Arnol'd Conjecture \cite{Ar1978} (see also \cite{Au13}, \cite{Ho12} 
for discussions in a historical context).  The first breakthrough on the 
conjecture was by Conley and Zehnder \cite{CoZe1983}, who
proved it for the $2n$\--torus (a proof
using generating functions was later given
by Chaperon \cite{Ch1984}).  The second breakthrough
was by Floer \cite{Fl1988,Fl1989,Fl1989b,Fl1991}.  Related 
results were proven eg. by Hofer\--Salamon, Liu\--Tian, Ono, Weinstein \cite{Ho1985,HoSa1995, LiTi1998, On1995, We1983}.

\section{\textcolor{black}{Appendix B: Random Stationary versus Periodic}} \label{appendixB}

A universal choice for the $(\O,\cF,\tau)$ part of $\hat \O$ is as follows:
Let $\O$ to be set of ${\rm C}^1$ functions $\o=(\bar Q,\bar P):\cS\to\cS,$  equip $\O$ with the standard uniform metric, and choose $\cF$ to be the Borel $\sigma$-field associated with this metric. 
As for $\tau$, simply define $\tau_a\o(q,p)=\o(q+a,p)$. Given $\o\in\O$, we set
 \[
F(q,p;\o)=(q,0)+\o(q,p)=\big(q+\bar Q(q,p),\bar P(q,p)\big).
\]
The conditions (i)\--(iii) of Definition~\ref{def1} really mean that $\bP$ is concentrated on the space functions $\o$
such that $\bar P(q,\pm 1)=\pm 1$, $\pm \bar Q(q,\pm 1)>0,$ $1+\bar Q_q(q,\pm 1)>0$, and that
 the corresponding function $F$ is area preserving. By choosing different $\tau$-invariant and ergodic measures
$\bP$ on $\O$, we are selecting different notions of {\em{generic}} area-preserving twist maps. 

We now discuss some examples of $\bP$ in order to explain the scope of our main theorems. 

\medskip\noindent
{\bf{(i)}} ({\em{Periodic Setting}})
As the simplest example, take a $q$-periodic $\tilde\o$ and assume that $\bP$ is concentrated on the translates
of $\tilde\o$. That is, $\bP(\G(\tilde\o))=1$ where
\begin{equation}\label{eq1.4}
\G(\tilde\o)=\left\{\tau_a\tilde\o\ :\ a\in\bR\right\}.
\end{equation}
Since $\tilde\o$ is periodic of period $1$ in the $q$ variable, the set $\G(\tilde\o)$ is homeomorphic to the circle.
(Here were are thinking of circle as the interval $[0,1]$ with $0=1$.)
Since $\bP$ is $\tau$-invariant, $\bP$ can only be the push forward of the Lebesgue measure under the map
$a\mapsto \tau_a\tilde\o$. Now any almost sure statement for the fixed points of $F(\cdot,\cdot;\o)$
 in this case boils down to an analogous statement for the map $\tilde F(q,p)=(q,0)+\tilde\o(q,p)$. This is because if $(q_0,p_0)$ is a fixed point for 
$(q,0)+\tau_a\tilde\o(q,p)$, then $(q+a,p)$ is a fixed point for $\tilde F(q,p)$.
In other words, our stochastic model \emph{coincides} with the classical setting of Poincar\'e-Birkhoff in this case. 
\qed

\medskip\noindent
{\bf{(ii)}} ({\em{Quasi-periodic Setting}}) Pick a function 
$$\bar\o=(\bar Q,\bar P):\bT^k\times [-1,-1]\to\bR\times[-1,1],$$
and a vector $v$ that satisfies the condition of Example~\ref{expr} (i).
Let $\tilde \o(q,p)=\bar\o(qv,p)$ and define $\G(\tilde\o)$ as in \eqref{eq1.4}. Note that if $k>1$, the set 
$\G(\tilde\o)$ is not closed, and its topological closure $\G'(\tilde\o)$ consists of
functions of the form
\[
\o^b(q,p)=\bar\o(b+qv,p),
\]
with $b\in\bT^k$.
(Here we regard $\bT^k$ as $[0,1]^k$ with $0=1$, and $b+qv$ is a {\em{Mod}}\ $1$ 
summation.)
Assume that $\bP$ is concentrated on the set $\G'(\tilde\o)$.
Again, since $\bP$ is $\tau$-invariant, the pull-back of $\bP$ with respect to the transformation
$b\in\bT^k\mapsto \o^b$ can only be the uniform measure on $\bT^k$. Hence, an almost sure statement regarding the fixed points of  $F(\cdot,\cdot;\o)$ is equivalent to an analogous statement for the map
$F^b(q,p)=(q,0)+\o^b(q,p)$, for almost all $b\in\bT^k$. In other words, our main result does not guarantee
the existence of fixed points for a given quasi-periodic map $\tilde F(q,p)=
(q,0)+\tilde\o(q,p)$. Instead we show that 
 almost all $k$-dimensional translates of $\tilde F$, namely $(F^b:b\in\bT^k)$,  possess fixed points as we stated in Theorems~\ref{epbt}\--\ref{epbt0}.
\qed

\medskip\noindent
{\bf{(iii)}} ({\em{Almost Periodic Setting}}) 
Given a function $\tilde\o\in \O$, let us assume that the corresponding $\G(\tilde\o)$ is precompact.
We write $\G'(\tilde\o)$ for the topological closure of $\G(\tilde\o)$.
 By the classical theory of almost periodic functions,  the set $\G'(\tilde\o)$ can be turned to a compact topological group and for $\bP$, we may choose a normalized {\em{Haar}} measure on $\G'(\tilde\o)$.
Again, our main results only guarantee
the existence of fixed points for almost all translates of the almost periodic  map $\tilde F(q,p)$.
\qed

As our above examples indicate, we prove the existence of fixed points for almost all translates of
almost periodic area preserving twist maps provided that certain additional conditions 
(as described in Theorems~\ref{epbt}\--\ref{epbt0}) are satisfied. Indeed one of the main points of our work is that we can 
go beyond almost-periodic seeing. In the random stationary setting, we may start with a function $\tilde \o$
such that the corresponding $\G'(\tilde\o)$ is not compact and may not have a group structure. Instead we 
may insist on the existence of an ergodic translation invariant
 measure that is concentrated on $\G'(\tilde\o)$.
 Even the last requirement can be relaxed and our measure $\bP$ may not be concentrated on $\G'(\tilde\o)$
for some $\tilde\o\in\O$. The measure $\bP$ in some sense plays
 the role  of the normalized Haar measure in our third example above. 

We now describe two important examples of area preserving random twist 
maps that go beyond the almost-periodic setting.

\medskip\noindent
{\bf{(iv)}} ({\em{Monotone Twist Maps}}) 
As we learned in Subsection~\ref{mono}, we can construct arbitrary monotone twist maps from their generating functions.
On the account of Theorem~\ref{prop9.2}, let us  write $\O_0$ for the space of functions $\o':\bR\times[-1,1]\to\bR$
such that $o'(q,0)=0$, $\o'(q,p)>0$ for $p>0$, and 
\[
\eta(q;\o')=\inf\{a\ |\ \o'(q,a)=2\}<\i,
\]
for every $q$.  we then set $\s(q;\o')=\eta(q;\o')-\frac 12\int_0^{\eta(q;\o')} \o'(q,a){\rm d}a$ and
\begin{align*}
Q^-(q;\o')&=-\s(q;\o'),\ \ \ \ G(q,Q;\o')=\o'(q,Q-q-Q^-(q;\o'));  \\
\cG(q,Q;\o')&=\int_{q+Q^-(q;\o')}^{Q}\o'(q,a){\rm d}a-(Q-q).
\end{align*}
We write $\O'$ for the space of $\o'\in\O_0$ such that $G_q(q,Q;\o')<0$ for all $(q,Q)$.
By Theorem~\ref{prop9.2}, there exists a unique $\o(q,p;\o')$ such that if $F(q,p;\o')=(q,0)+\o(q,p;\o')$
and $\o'\in\O'$, then  
$
F(q,-\mathcal{G}_q(q,Q;\o'))=(Q,\mathcal{G}_Q(q,Q;\o')).
$
We define $\tau_a\o'(q,p)=\o'(q+a,p)$ and start with an arbitrary $\tau$-invariant ergodice
probability measure $\bP'$ such that $\bP'(\O')=1$.  The push forward of $\bP'$ under the 
transformation $\cF(\o')= \o(\cdot,\cdot;\o')$ yields a probability measure $\bP$ that is concentrated on those
$\o\in\O$ such that the corresponding $F$ is a monotone twit map.  We note that the map $\cF$ is continuous
and $\cF(\tau_a\o')=\tau_a\cF(\o')$. From this we learn that if $\o'$ is almost periodic, then $\cF(\o')$ is also
almost periodic. Though we can readily construct examples of $\bP'$ that is not concentrated on the space of almost periodic functions.

\medskip\noindent
{\bf{(v)}} ({\em{Hamiltonian Systems}})
Let $\O_1$ denote the set ${\rm C}^2$ (Hamiltonian) functions $H(q,p,t)$ such that $\pm H_p(q,\pm 1,t)>0$
and $H_q(q,\pm 1,t)=0$. Given $H\in\O_1$, set $\tau_aH(q,p,t)=H(q+a,p,t)$ and write $\phi_t^H(q,p)$ for its flow. We then define $F^t(q,p;H)=\phi_t^H(q,p)$ and $\o^t(q,p;H)=F^t(q,p;H)-(q,0)$. 
We can readily show that 
$F^t$ is a twist map and that $\o^t(q,p;\tau_aH)=\o^t(q+a,p;H)=\tau_a\o^t(q,p;H)$.
 Finally any $\tau$-invariant ergodic probability measure $\bQ$ on $\O_1$ yields a probability measure
$\bP^t$ on $\O$ and the twist maps $F^t$ are isotopic to identity. As a concrete example, take any $\bar H\in\O_1$ of compact support and given a discrete set $\a=\{q_i\ |\ i\in \bZ\}$,
define
\[
H(q,p;\a)=\sum_i\bar H(q-q_i,p).
\]
If we select $\a$ randomly according to a stationary point process (such as Poisson point process as was
described in Example~\ref{expr}), then the law of $H(q,p;\a)$ is an example of a $\tau$-invariant ergodic measure on the space $\O_1.$
\qed 

\bigskip
\bigskip
\medskip
\medskip
\medskip

{
\paragraph{\emph{Acknowledgements}.}   We thank Atilla Yilmaz for discussions, 
and Barney Bramham and Kris Wysocki for
comments on a preliminary version.  

\medskip

 AP is supported in part by NSF CAREER DMS-1055897,
NSF Grant DMS-0635607, a J. Tinsely Oden Faculty Fellowship from the University of
Texas, and an Oberwolfach Leibniz Fellowship. He
is grateful to the Institute for Advanced Study for providing an excellent research environment for
his research (December 2010\,\---\,\,September 2013),  and to Helmut Hofer for fruitful discussions on many topics.  During the
last stages of in the preparation of this paper, he was hosted by Luis Caffarelli,
Alessio Figalli, and J. Tinsley Oden at the Institute for Computational Sciences and Engineering in Austin, 
and he is  thankful for the hospitality. FR is supported in part by NSF Grant DMS-1106526. 

\medskip

Parts of this paper were written during visits of AP to UC Berkeley and  FR to the
Institute for Advanced Study and Washington University 
between 2010 and 2013. Finally, thanks to Tejas Kalelkar for help drawing some of
the figures.}

\bigskip
\bigskip

{ \paragraph{{\bf \emph{Dedication}}.}
The authors dedicate this article to Alan Weinstein, whose deep insights in so many areas of geometry are a continuous source of inspiration.}

{\small
\noindent
\\
{{\'A}lvaro Pelayo}\\
University of California, San Diego\\ 
Department of Mathematics \\
9500 Gilman Dr \#0112\\
La Jolla, CA, USA.\\
{\em E\--mail}: \texttt{alpelayo@math.ucsd.edu} \\

\medskip
\medskip

\noindent
{Fraydoun Rezakhanlou}\\
University of California, Berkeley\\
Department of Mathematics\\
 Berkeley, CA, 94720-3840, USA. \\
{\em E\--mail}: \texttt{rezakhan@math.berkeley.edu}\\

 \end{document}